\documentclass[reqno, 10pt]{amsart}
\allowdisplaybreaks

\usepackage[english]{babel}
\usepackage[latin1]{inputenc}
\usepackage[T1]{fontenc}
\usepackage{amssymb,amsmath,amstext,amsfonts,amsthm}
\usepackage{mathtools}
\usepackage{verbatim}
\usepackage{relsize}
\usepackage{graphicx,color}
\usepackage[dvipsnames]{xcolor}
\usepackage[all]{xy}
\usepackage[colorlinks=true, urlcolor=black, linkcolor=blue, citecolor=blue]{hyperref}
\usepackage{fancyvrb}
\usepackage{tikz-cd}
\usepackage{comment}
\usepackage{hhline}
\usepackage{epigraph}
\usepackage[left=25mm, right=25mm, top=30mm, bottom=30mm, includefoot, includehead]{geometry}

\numberwithin{equation}{section}

\theoremstyle{plain}
\newtheorem{Theorem}{Theorem}
\newtheorem*{Theorem*}{Theorem}

\newtheorem{Lemma}[Theorem]{Lemma}
\newtheorem{Proposition}[Theorem]{Proposition}

\theoremstyle{definition}
\newtheorem{Definition}[Theorem]{Definition}
\newtheorem{Example}[Theorem]{Example}
\newtheorem*{notation}{Notation}

\theoremstyle{remark}
\newtheorem{Remark}[Theorem]{Remark}
\newtheorem{Question}[Theorem]{Question}

\newcommand{\CC}{\mathbb{C}}
\newcommand{\NN}{\mathbb{N}}

\newcommand{\PP}{\mathbb{P}}
\newcommand{\FF}{\mathbb{F}}
\newcommand{\RR}{\mathbb{R}}
\newcommand{\mR}{\mathsmaller{\mathbb{R}}}
\newcommand{\mC}{\mathsmaller{\mathbb{C}}}
\def\xx{{\bf x}}
\def\yy{{\bf y}}
\def\zz{{\bf z}}
\def\ww{{\bf w}}
\def\cc{{\bf c}}

\def\vv{{\bf v}}
\def\nn{{\bf n}}
\def\mm{{\bf m}}

\def\bdelta{{\boldsymbol{\delta}}}
\def\bomega{{\boldsymbol{\omega}}}
\def\balpha{{\boldsymbol{\alpha}}}
\def\codim{{\rm codim}}
\def\C{{{\mathbb C}}}

\makeatletter
\@namedef{subjclassname@2020}{%
	\textup{2020} Mathematics Subject Classification}
\makeatother

\author{Zahra Shahidi}
\address{Universit\`{a} di Firenze, Dipartimento di Matematica e Informatica, Viale Morgagni 67/A, 50134 Firenze, Italy}
\email{z.shahidi80@gmail.com}
\author{Luca Sodomaco}
\address{Department of Mathematics and Systems Analysis, Aalto University, Espoo, Finland}
\email{luca.sodomaco@aalto.fi}
\author{Emanuele Ventura}
\address{Politecnico di Torino, Dipartimento di Scienze Matematiche ``G. L. Lagrange'', Corso Duca degli Abruzzi 24, 10129 Torino, Italy}
\email{emanuele.ventura@polito.it, emanueleventura.sw@gmail.com}


\title{Degrees of Kalman varieties of tensors}

\begin{document}

\maketitle

\begin{abstract}
Kalman varieties of tensors are algebraic varieties consisting of tensors whose singular vector $k$-tuples lay on prescribed subvarieties. They were first studied by Ottaviani and Sturmfels in the context of matrices. We extend recent results of Ottaviani and the first author to the partially symmetric setting. We describe a generating function whose coefficients are the degrees of these varieties and we analyze its asymptotics, providing analytic results \`a la Zeilberger and Pantone. 
We emphasize the special role of isotropic vectors in the spectral theory of tensors and describe the totally isotropic Kalman variety as a dual variety. 
\end{abstract}

\section{Introduction}
Singular vector $k$-tuples are the partially symmetric tensor analog of singular vector pairs of rectangular matrices. Their definition is recalled in \S \ref{sec: preliminaries}. We refer to \cite{QZ} for background and applications. 

In this article, we tackle the problem of describing the set of partially symmetric tensors admitting a singular vector $k$-tuple $(\xx_1,\ldots,\xx_k)$, where (the equivalence class of) each entry  $\xx_i$ lies on a fixed irreducible subvariety $Z_i\subset \PP(V_i)$ of the $i$-th factor. 
This set is an algebraic variety in the projective space $\PP(S^{\omega_1} V_1\otimes\cdots\otimes S^{\omega_k} V_k)$ of partially symmetric tensors with $k$ factors called {\em generalized Kalman variety} of tensors.

Recently, the spectral theory of tensors along with its connections to pure and applied algebraic geometry and combinatorics is witnessing several interesting results; see e.g. \cite{BGV, Turatti} for recent geometric results. Despite this progress, algebraic varieties providing a clean geometric picture of spectral properties of tensors await to be described. Kalman varieties are therefore central objects in this context. 
This paper is a contribution to this circle of ideas from the perspective of degrees of varieties, their generating functions and asymptotics. 

The name {\em Kalman variety} was first introduced by Ottaviani and Sturmfels in \cite{OSt} to indicate the variety of square matrices possessing at least one eigenvector on a fixed linear subspace. They determined its codimension, degree and studied its singular locus.
Thereafter, Sam \cite{Sam} and Huang \cite{Huang} determined their defining equations. More recently, Ottaviani and the first author \cite{OSh} rephrased the original setting for singular vector pairs, extending it to the case of singular vector $k$-tuples.
Their Kalman variety is the variety of tensors having a singular vector $k$-tuple $(\xx_1,\ldots,\xx_k)$, where (the equivalance class of) the first entry $\xx_1$ lies on a fixed linear subspace $L\subset \PP(V_1)$. 

Our point of departure is \cite[Theorem 1.2]{OSh}, where the authors showed that the codimension of this Kalman variety is equal to the codimension of $L\subset\PP(V_1)$. Moreover, they established an elegant formula for the degrees of the Kalman variety for symmetric and non-symmetric tensors. In the latter case, letting $n_i=\dim(V_i)$ and $\delta=\codim(L)$, the desired degree is the coefficient of the monomial $h^{\delta}t_{1}^{n_1-\delta-1}\prod_{i\geq 2}t_{i}^{n_{i}-1}$ in the polynomial
\[
\prod_{i=1}^{k}\frac{(\widehat{t_i}+h)^{n_{i}}-t_i^{n_i}}{(\widehat{t_i}+h)-t_i}\,,\quad\widehat{t_i}\coloneqq\left(\sum_{j=1}^{k}t_j\right)- t_i\,.
\]
This expression is similar to the formula for the number of singular vector $k$-tuples of a general tensor.
The latter quantity coincides with another well-known metric invariant of an algebraic variety, called the {\it ED degree} of the Segre variety $\PP(V_1)\times\cdots\times\PP(V_k)$.
For more details on ED degrees of algebraic varieties, we refer to \cite{DHOST}.
Ottaviani and Friedland \cite[Theorem 1]{FO} computed the above ED degree using Chern classes of a suitable vector bundle on the Segre variety.

The aforementioned results establish the {\it enumerative} nature of the subject.
When enumerative structures appear, it is a natural problem to determine a generating function whose coefficients are the counted quantities.
Generating functions are tremendously useful tools to have a global picture of the enumerated objects. See \cite{MacMahon, Stanley, Wilf} for fascinating introductions to this topic along with its applications to combinatorics and analysis.
Zeilberger \cite{EZ} found a generating function for the Friendland-Ottaviani's formula and hence for the ED degrees. The asymptotic behavior was then analyzed by Pantone \cite{Pan}.
Zeilberger's generating function is strikingly similar to the generating function of degrees of hyperdeterminants. An asymptotic analysis of the latter was performed in \cite[Theorem 3.8]{OSV}.

Our first contribution is a generalization of \cite[Theorem 1.2]{OSh} and \cite[Theorem 12]{FO}. 
More specifically, we determine the degrees of generalized Kalman varieties of partially symmetric tensors.

\begin{Theorem}\label{thm: OSh, partially symmetric case}
For every $i\in[k]$, let $Z_i\subset \PP(V_i)$ be an irreducible projective variety of codimension $\delta_i$ and let $Z=\prod_{i=1}^k Z_i$. We assume that each $Z_i$ is not contained in the isotropic quadric $Q_i\subset\PP(V_i)$.
Let $\bdelta=(\delta_1,\ldots,\delta_k)$. The generalized Kalman variety
\[
\kappa_{\nn, \bomega}(Z)\coloneqq\{T\in\PP(S^{\bomega}V) \mid\mbox{$T$ has a singular vector $k$-tuple $([\xx_1], \ldots, [\xx_k])\in Z$}\}
\]
is connected of codimension $\delta\coloneqq\sum_{i=1}^k\delta_i$. The degree of $\kappa_{\nn, \bomega}(Z)$ is $d(\nn,\bdelta,\bomega)\prod_{i=1}^k\deg(Z_i)$, where $d(\nn,\bdelta,\bomega)$ is the coefficient of the monomial $h^{\delta}\prod_{i=1}^kt_{i}^{n_{i}-\delta_i-1}$ in the polynomial
\[
\prod_{i=1}^{k}\frac{(\widehat{t_i}+h)^{n_{i}}-t_i^{n_i}}{(\widehat{t_i}+h)-t_i}\,,\quad\widehat{t_i}\coloneqq\left(\sum_{j=1}^{k}\omega_j t_j\right)- t_i\,.
\]
\end{Theorem}

We study separately the case when each subvariety $Z_i$ coincides with the isotropic quadric $Q_i$. The corresponding Kalman variety is called {\em totally isotropic Kalman variety}. In Theorem \ref{thm: codim degree Kalman partially symmetric isotropic}, we describe the totally isotropic Kalman variety as a dual variety of a specific Segre-Veronese variety.
This specializes to Theorem \ref{thm: codim degree Kalman symmetric isotropic} in the symmetric case; this last result is related to \cite[Proposition 2.10]{BGV}.

Our second main result furnishes a Zeilberger-type \cite[Proposition 1]{EZ} generating function for the degrees of generalized Kalman varieties in the case $\bdelta=(\delta,0,\ldots,0)$.

\begin{Theorem}\label{thm: generating function, partially symmetric case}
Keep the assumptions of Theorem \ref{thm: OSh, partially symmetric case}. If $\bdelta=(\delta,0,\ldots,0)$ for some $\delta\ge 0$, the generating function for the coefficients $d(\nn,\delta,\bomega)$ is
\begin{equation}\label{eq: generating function, partially symmetric}
\sum_{\nn\in\mathbb{N}^k}\sum_{\delta=0}^{\infty} d(\nn,\delta,\bomega)\,\xx^\nn y^\delta\,=\,\frac{1}{H_{\bomega}(\xx,y)}\prod_{i=1}^k\frac{x_i}{1-x_i}
\end{equation}
where
\begin{equation}\label{eq: def H, partially symmetric}
H_{\bomega}(\xx,y)\coloneqq -y\,x_1\prod_{i=2}^k(1+x_i)+\prod_{i=1}^k(1+ x_i)- \sum_{j=1}^k \omega_jx_j\prod_{i \neq j}(1+ x_i)\,.
\end{equation}
\end{Theorem}

Our third result offers an asymptotic study of the degrees $d(\nn,\bdelta,\bomega)$ in three different regimes: first, when $n_i\to\infty$ for a single index $i$. Secondly, in the binary format $n_1=\cdots=n_k=2$ as $k\to\infty$. Finally, in the hypercubical format $n_1=\cdots=n_k=n$ for $n\to\infty$. The last asymptotic study is more involved and leads to the following result, which agrees with \cite[Theorem 1.3]{Pan} for $\delta=0$ and $\omega=1$. We refer to the first two paragraphs of Section \ref{sec: preliminaries} for the preliminary notations and definitions used.

\begin{Theorem}\label{thm: deg Kalman approx}
Consider the factor $d(n,\delta,\omega)$ in the degree of the generalized Kalman variety $\kappa_{n{\bf 1}, \omega{\bf 1}}(Z)$. Assume that either $k=2$ and $w \geq 2$ or $k \geq 3$. Then asymptotically, for $n\to +\infty$, 
\[
d(n,\delta,\omega) = \frac{(\omega k-1)^{k-1}}{(2\pi)^\frac{k-1}{2}\,(\omega k)^{\frac{k-2}{2}}(\omega k-2)^\frac{3k-1}{2}}\left(\frac{\omega k}{\omega k-1}\right)^\delta\frac{(\omega k-1)^{kn}}{n^{\frac{k-1}{2}-\delta}} \left[1+O\left(\frac{1}{n}\right)\right]\,.
\]
\end{Theorem}

In Definition \ref{def: kalmanstrata}, we introduce the {\it Kalman strata}. These are, roughly speaking, the building blocks of generalized Kalman varieties,  and are meant to take into account the presence of {\it isotropic} vectors in the singular vector tuple. The proof of Theorem \ref{thm: OSh, partially symmetric case} relies on determining the linear space of tensors having a singular vector $k$-tuple over a specific Kalman stratum. This is achieved in Theorem \ref{thm: fibers}. This result is interesting on its own right, giving an intrinsic description of the fibers in terms of orthogonal spaces. Another consequence of Theorem \ref{thm: fibers} is Proposition \ref{prop: irredsigmax}, where we show the irreducibility of the {\it spectral variety} (see Definition \ref{def: spectralvariety}), which is a natural incidence correspondence in this setting. 

In \S \ref{sec: kalmanofsymmetricvectortuples}, we introduce Kalman varieties of symmetric singular vector $k$-tuples and derive their codimensions. The degrees of these new interesting varieties seem more challenging even in the case of matrices. We leave their determination as an open question.

\section*{Acknowledgements}
\begin{small}
The idea of this project was conceived during Shahidi's postdoc at Universit\`{a} di Firenze. We thank Giorgio Ottaviani for very useful discussions and  encouragement. 
We thank Jan Draisma for explaining to us a way of deriving the equations in Example \ref{ex: deg7example}. The first author thanks Dr. Alireza Firoozfar and Dr. Mohsen Afsharchi for their support.
The second author is partially supported by the Academy of Finland Grant 323416. During most of the preparation of the manuscript, the third author was a postdoc at Universit\"{a}t Bern, supported by Vici Grant 639.033.514 of Jan Draisma from the  Netherlands Organisation for Scientific Research.
We thank two anonymous referees 
for their useful comments and questions that also helped to improve the presentation.
\end{small}

\section{Kalman varieties of tensors}\label{sec: preliminaries}

\begin{notation}
Throughout the paper, if not specified we denote by $\bf{j}$ a vector $(j_1,\ldots,j_k)$ of $k$ variables, while $\xx^\mm$ stands for the monomial $x_1^{m_1}\cdots x_k^{m_k}$. We often use the shorthand $[k]$ to denote the set of indices $\{1,\ldots,k\}$. Define ${\bf 1} = (1,\ldots,1)\in \NN^k$ and, for $m\in \NN$, let $m{\bf 1} = (m,\ldots, m)\in \NN^k$. 
\end{notation}

For every $i\in[k]$ we consider an $n_i$-dimensional vector space $V_i$ over the field $\FF=\RR$ or $\FF=\C$. If $\FF=\RR$, then we prefer the notation $V_i^\mR$.
Let $\bomega=(\omega_1,\ldots,\omega_k)$ be a vector of nonnegative integers.
For each $i\in[k]$, let $S^{\omega_i}V_i$ be the $\omega_i$-th symmetric power of $V_i$, as a subspace of the tensor product $V_i^{\otimes \omega_i}$.
Moreover, we denote by $S^{\bomega}V$ the tensor product $\bigotimes_{i=1}^k S^{\omega_i}V_i$.
This is the space of {\em partially symmetric tensors} of format $n_1^{\times\omega_1}\times\cdots\times n_k^{\times\omega_k}$.
Every element of $S^{\bomega}V$ is a linear combination of {\em decomposable partially symmetric tensors}, that is, tensors of the form $T=\xx_1^{\omega_1}\otimes\cdots\otimes\xx_k^{\omega_k}$ for some vectors $\xx_i\in V_i$.
On each projective space $\PP(V_i)$ we fix a smooth projective quadric hypersurface $Q_i=\mathcal{V}(q_i)$, where $q_i$ is the homogeneous polynomial in $\C[x_{i,1},\ldots,x_{i,n_i}]_2$ associated to a positive definite real quadratic form $q_i^\mR\colon V_i^\mR\to\RR$. We refer to $Q_i$ as the {\em isotropic quadric} in the $i$-th factor $\PP(V_i)$.
Finally, we denote by $\PP$ the product $\prod_{i=1}^k\PP(V_i)$ and by $\Pi_{\nn,\bomega}$ the product $\PP(S^\bomega V)\times\PP$.

\begin{Definition}\label{def: Frobenius inner product}
The {\it Frobenius (or Bombieri-Weyl) inner product} of two complex decomposable tensors $T = \xx_1^{\omega_1}\otimes\cdots\otimes \xx_k^{\omega_k}$ and $T' = \yy_1^{\omega_1}\otimes\cdots\otimes \yy_k^{\omega_k}$ of $S^{\bomega}V$ is
\begin{equation}\label{eq: Frobenius inner product for tensors}
q_F(T, T')\coloneqq q_1(\xx_1, \yy_1)^{\omega_1}\cdots q_k(\xx_k, \yy_k)^{\omega_k}\,,
\end{equation}
and it is naturally extended to every vector in $S^{\bomega}V$.
We identify all the vector spaces with their duals using the Frobenius inner product.
\end{Definition}

\begin{Definition}\label{def: singular vector tuple}
Let $T\in S^{\bomega}V$. A {\em singular vector $k$-tuple} of $T$ is a $k$-tuple $(\xx_1,\ldots,\xx_k)$ of nonzero vectors $\xx_i\in V_i$ such that
\begin{equation}\label{eq: def singular vector tuple matrix}
\mathrm{rank}
\begin{pmatrix}
T(\xx_1^{\omega_1}\otimes\cdots\otimes \xx_i^{\omega_i-1}\otimes\cdots\otimes \xx_k^{\omega_k})\\
\xx_i
\end{pmatrix}
\le 1\quad\forall\,i\in[k]\,,
\end{equation}
where $T(\xx_1^{\omega_1}\otimes\cdots\otimes \xx_i^{\omega_i-1}\otimes\cdots\otimes \xx_k^{\omega_k})$ is the tensor contraction of $T$ with respect to $\xx_1^{\omega_1}\otimes\cdots\otimes \xx_i^{\omega_i-1}\otimes\cdots\otimes \xx_k^{\omega_k}$.
If $\omega_i=1$ for some $i\in[k]$, then $\xx_i^{\omega_i-1}=\xx_i^0=1$ as an element of $S^0V_i=\FF$.
If we interpret $T$ as a multi-homogeneous polynomial in the coordinates of each vector $\xx_i$, then the previous tensor contraction corresponds the gradient $\nabla_iT$ with respect to the coordinates of $\xx_i$.\\
A singular vector $k$-tuple $(\xx_1,\ldots,\xx_k)$ is {\em normalized} if $q_i(\xx_i)=1$ for all $i\in[k]$. A singular vector $k$-tuple $(\xx_1,\ldots,\xx_k)$ is {\em isotropic} if $q_i(\xx_i)=0$ for some $i\in[k]$. When $k=1$, $\nn = (n)$ and $\bomega = (\omega)$, a vector $\xx\in V$ satisfying \eqref{eq: def singular vector tuple matrix} is called an {\em eigenvector} of the symmetric tensor $T\in S^{\omega}V$.
\end{Definition}

\begin{Definition}\label{def: singular values}
Let $T\in S^{\bomega}V$ and let $(\xx_1,\ldots,\xx_k)$ be a singular vector $k$-tuple of $T$. For every $i\in[k]$, the value $\sigma_i\in\C$ such that
\begin{equation}\label{eq: svt singular values}
T(\xx_1^{\omega_1}\otimes\cdots\otimes \xx_i^{\omega_i-1}\otimes\cdots\otimes \xx_k^{\omega_k}) = \sigma_i\,\xx_i
\end{equation}
is called the {\em $i$-th singular value} of $(\xx_1,\ldots,\xx_k)$.
\end{Definition}

\begin{Remark}\label{rmk: equal singular values}
Let $T\in S^{\bomega}V$ and let $(\xx_1,\ldots,\xx_k)$ be a singular vector $k$-tuple of $T$ with the associated $k$-tuple $(\sigma_1,\ldots,\sigma_k)$ of singular values. For every $i\in[k]$, the identity \eqref{eq: svt singular values} can be interpreted as the identity in $(V_i)^*$
\[
q_F(T, \xx_1^{\omega_1}\otimes\cdots\otimes \xx_i^{\omega_i-1}\cdot\_\otimes\cdots\otimes \xx_k^{\omega_k}) = \sigma_i\,q_i(\xx_i,\_)\,.
\]
Evaluating the left-hand side of the previous identity at the vector $\xx_i$ yields the number $q_F(T, \xx_1^{\omega_1}\otimes\cdots\otimes \xx_k^{\omega_k})$ which does not depend on the specific index $i$. Thus we have $\sigma_1\,q_1(\xx_1,\xx_1)=\cdots=\sigma_k\,q_k(\xx_k,\xx_k)$. Therefore
\begin{itemize}
    \item[$(i)$] $\sigma_1=\cdots=\sigma_k$ if $(\xx_1,\ldots,\xx_k)$ is normalized. The common value $\sigma\coloneqq \sigma_1=\cdots=\sigma_k$ is often called the {\em singular value} of the normalized singular vector $k$-tuple $(\xx_1,\ldots,\xx_k)$.
    \item[$(ii)$] $\sigma_j=0$ for all $j\in[k]$ such that $q_j(\xx_j,\xx_j)\neq0$ if $(\xx_1,\ldots,\xx_k)$ is isotropic. 
\end{itemize}
\end{Remark}

If $(\xx_1,\ldots,\xx_k)$ is a singular vector $k$-tuple of $T\in S^{\bomega}V$, then it is immediate to check that, for every tuple $(\lambda_1,\ldots,\lambda_k)$ of nonzero complex numbers, the $k$-tuple $(\lambda_1\,\xx_1,\ldots,\lambda_k\,\xx_k)$ is also a singular vector $k$-tuple of $T$. For this reason, we consider $\PP=\prod_{i=1}^k \PP(V_i)$ and say that $([\xx_1],\ldots,[\xx_k])\in\PP$ is a singular vector $k$-tuple of $T$ if its representative $(\xx_1,\ldots,\xx_k)$ is.

\begin{Definition}
For every $i\in[k]$, we denote by $\mathrm{SO}(V_i)$ the (complex) {\em special orthogonal group} of automorphisms of $V_i$ with determinant 1 that preserve the bilinear product in $V_i$ associated to $q_i$. 
\end{Definition}

\begin{Remark}\label{rmk: actionSO}
The notion of singular vector $k$-tuple is $\mathrm{SO}(V_1)\times\cdots\times\mathrm{SO}(V_k)$-equivariant. More explicitly, given a singular vector $k$-tuple $(\xx_1,\ldots,\xx_k)$ of $T$ and an element $g_i\in \mathrm{SO}(V_i)$ for all $i\in[k]$, then $(g_1(\xx_1),\ldots,g_k(\xx_k))$ is a singular vector $k$-tuple of $(g_1,\ldots, g_k)\cdot T$. This is because the contraction (or scalar product) in the $i$-th factor is preserved by the action of $\mathrm{SO}(V_i)$. 
\end{Remark}

\begin{Lemma}\label{lem: characterize singular vector tuples}
Consider a tuple $([\xx_1],\ldots, [\xx_k])\in\PP$. For all $i\in[k]$, define the subspace
\begin{align}\label{eq: def W_i}
\begin{split}
W_i &\coloneqq\xx_1^{\omega_1}\otimes\cdots\otimes\xx_i^{\omega_i-1}\langle\xx_i\rangle^\perp\otimes\cdots\otimes\xx_k^{\omega_k}\\
&=\{\xx_1^{\omega_1}\otimes\cdots\otimes\xx_i^{\omega_i-1}\cdot\ww_i\otimes\cdots\otimes\xx_k^{\omega_k}\mid\ww_i\in\langle\xx_i\rangle^\perp\}\subset S^{\bomega}V\,,
\end{split}
\end{align}
where $\langle\xx_i\rangle^\perp$ is the orthogonal complement of $\langle\xx_i\rangle$ in $V_i$ with respect to the fixed quadratic form $q_i$.\\
Given a tensor $T\in S^{\bomega}V$, we have that $([\xx_1],\ldots, [\xx_k])\in\PP$ is a singular vector $k$-tuple of $T$ if and only if $T\in\left(W_1+\cdots+W_k\right)^\perp$, where in this case the sign $\perp$ denotes the orthogonal complement with respect to the Frobenius inner product in $S^{\bomega}V$.
\end{Lemma}

\begin{proof}
By definition, the tuple $([\xx_1],\ldots, [\xx_k])\in\PP$ is a singular vector $k$-tuple of $T$ if and only if the relations in \eqref{eq: def singular vector tuple matrix} hold. For a fixed $i\in[k]$, the corresponding relation in \eqref{eq: def singular vector tuple matrix} is equivalent to $T\in W_i^\perp$. Hence $([\xx_1],\ldots, [\xx_k])\in\PP$ is a singular vector $k$-tuple of $T$ if and only if $T\in W_1^\perp\cap\cdots\cap W_k^\perp$, and one has the equality $W_1^\perp\cap\cdots\cap W_k^\perp=\left(W_1+\cdots+W_k\right)^\perp$.
\end{proof}

\noindent We introduce an incidence correspondence that will play a fundamental role in the upcoming proofs.

\begin{Definition}\label{def: spectralvariety}
Let $\Pi_{\nn,\bomega}=\PP(S^{\bomega}V) \times \PP$. The {\em spectral variety of type $(\nn,\bomega)$} is 
\begin{equation}\label{eq: def Sigma}
\Sigma_{\nn,\bomega} \coloneqq \{ ([T], [\xx_1],\ldots, [\xx_k])\in \Pi_{\nn,\bomega} \mid \text{$([\xx_1],\ldots,[\xx_k])$ is a singular vector $k$-tuple for $T$}\}\,.
\end{equation}
\end{Definition}

\begin{Remark}\label{rmk: spectral is Zar closed}
Note that $\Sigma_{\nn,\bomega}$, as described set-theoretically above, is closed in the Zariski topology: equipped with its reduced structure, it is the subvariety whose ideal is the radical ideal of
\[
(J_1+\cdots+J_k)\colon\left(\prod_{i=1}^k\langle \xx_i\rangle\right)^\infty\subset\C[\xx_1,\ldots,\xx_k]\,,
\]
where $\langle\xx_i\rangle$ is the ideal generated by the coordinates of $\xx_i$ and
\[
J_i=\left\langle\mbox{$2\times 2$ minors of }
\begin{pmatrix}
T(\xx_1^{\omega_1}\otimes\cdots\otimes\xx_i^{\omega_i-1}\otimes\cdots\otimes\xx_k^{\omega_k})\\
\xx_i
\end{pmatrix}
\right\rangle\quad\forall\,i\in[k]\,.
\]
\end{Remark}

\noindent In the following, we will often consider the diagram
\begin{equation}\label{eq: diagram spectral variety}
\begin{tikzcd}
& \Sigma_{\nn,\bomega} \arrow{dl}[swap]{\alpha} \arrow{dr}{\beta} & \\
\PP(S^{\bomega}V) & & \PP\,,
\end{tikzcd}
\end{equation}
where $\alpha$ and $\beta$ denote the two projections along the two factors of $\Pi_{\nn,\bomega}$.

\begin{Theorem}\label{thm: fibers}
Let $\Sigma_{\nn,\bomega} \subset \Pi_{\nn,\bomega}$ be the spectral variety of type $(\nn,\bomega)$. The projection $\beta\colon\Sigma_{\nn,\bomega}\to\PP$ is surjective and every fiber of $\beta$ over a $k$-tuple $([\xx_1],\ldots, [\xx_k])$ with $r$ isotropic components is a projective subspace in $\PP(S^{\bomega}V)$ of codimension
\begin{equation}\label{eq: dim fiber q r isotropic factors}
\codim\left(\beta^{-1}([\xx_1],\ldots, [\xx_k])\right)=\sum_{i=1}^k(n_i-1)-\max\{0,r-1\}\,.
\end{equation}
\end{Theorem}

\begin{proof}
Consider a tuple $([\xx_1],\ldots, [\xx_k])\in\PP$. The fiber $\beta^{-1}([\xx_1],\ldots, [\xx_k])$ is isomorphic to the projective subspace of classes $[T]\in\PP(S^{\bomega}V)$ such that $([\xx_1],\ldots, [\xx_k])$ is a singular vector $k$-tuple of $T$. By Lemma \ref{lem: characterize singular vector tuples}, we have that 
\[
\beta^{-1}([\xx_1],\ldots, [\xx_k])\cong\PP(\left(W_1+\cdots+W_k\right)^\perp)\,,
\]
where the subspaces $W_i$ have been defined in \eqref{eq: def W_i}. Hence 
\[
\codim\left(\beta^{-1}([\xx_1],\ldots, [\xx_k])\right)=\dim(W_1+\cdots+W_k)\,.
\]
The proof goes by induction on $k$. First, we assume $k=2$. We have
\[
\dim(W_1+W_2)=\dim(W_1)+\dim(W_2)-\dim(W_1\cap W_2)=\sum_{i=1}^2(n_i-1)-\dim(W_1\cap W_2)\,,
\]
so it remains to compute $\dim(W_1\cap W_2)$ depending on the number of isotropic vectors in the pair $([\xx_1],[\xx_2])$. Consider nonzero elements $F_1\in W_1$ and $F_2\in W_2$, written explicitly as $F_1=\xx_1^{\omega_1-1}\cdot\ww_1\otimes\xx_2^{\omega_2}$ and $F_2=\xx_1^{\omega_1}\otimes\xx_2^{\omega_2-1}\cdot\ww_2$ for some vectors $\ww_1\in\langle\xx_1\rangle^\perp$ and $\ww_2\in\langle\xx_2\rangle^\perp$. Then necessarily $F_1=F_2$ only if $q_1(\xx_1)=q_2(\xx_2)=0$, in which case the equality is attained for $\ww_i\in\langle\xx_i\rangle\subset\langle\xx_i\rangle^\perp$, $i\in[2]$. In particular $\dim(W_1\cap W_2)=1$ and the identity \eqref{eq: dim fiber q r isotropic factors} is satisfied. Otherwise $W_1\cap W_2=\{0\}$ if at least one among $\xx_1$, $\xx_2$ is not isotropic, and \eqref{eq: dim fiber q r isotropic factors} is satisfied as well.

Now we assume that the identity \eqref{eq: dim fiber q r isotropic factors} is satisfied for $k-1$ factors among the given $k$. Without loss of generality, we apply the induction step to the first $k-1$ factors. In particular
\[
\dim(W_1+\cdots+W_k)=\dim(W_1+\cdots+W_{k-1})+\dim(W_k)-\dim\left[(W_1+\cdots+W_{k-1})\cap W_k\right]\,,
\]
so it remains to compute $\dim\left[(W_1+\cdots+W_{k-1})\cap W_k\right]$ depending on the number of isotropic vectors in the tuple $([\xx_1],\ldots, [\xx_k])$. Pick nonzero tensors $F_1\in W_1+\cdots+W_{k-1}$ and $F_2\in W_k$, which can be written explicitly as
\begin{align*}
F_1 &= \left(\sum_{i=1}^{k-1}\xx_1^{\omega_1}\otimes\cdots\otimes\xx_i^{\omega_i-1}\cdot\ww_i\otimes\cdots\otimes\xx_{k-1}^{\omega_{k-1}}\right)\otimes\xx_k^{\omega_k}\,,\ \ww_i\in\langle\xx_i\rangle^\perp\\
F_2 &= \xx_1^{\omega_1}\otimes\cdots\otimes\xx_{k-1}^{\omega_{k-1}}\otimes\xx_k^{\omega_k-1}\cdot\ww_k\,,\ \ww_k\in\langle\xx_k\rangle^\perp\,.
\end{align*}
If $q_k(\xx_k)\neq 0$, then necessarily $F_1\neq F_2$, so $(W_1+\cdots+W_{k-1})\cap W_k=\{0\}$. In this case the identity \eqref{eq: dim fiber q r isotropic factors} follows applying the induction step.
Otherwise $q_k(\xx_k)=0$ and in this case $F_1=F_2$ only if $\ww_k\in\langle\xx_k\rangle\subset\langle\xx_k\rangle^\perp$ and
\begin{equation}\label{eq: identity k-1 factors}
\sum_{i=1}^{k-1}\xx_1^{\omega_1}\otimes\cdots\otimes\xx_i^{\omega_i-1}\cdot\ww_i\otimes\cdots\otimes\xx_{k-1}^{\omega_{k-1}} = \xx_1^{\omega_1}\otimes\cdots\otimes\xx_{k-1}^{\omega_{k-1}}\,.
\end{equation}
Now suppose that $q_i(\xx_i)\neq 0$ for some $i\in[k-1]$, for simplicity $i=1$. Since $\langle\xx_1\rangle^\perp\oplus\langle\xx_1\rangle=\C^{n_1}$, we have the decomposition
\[
\left(\xx_1^{\omega_1-1}\langle\xx_1\rangle^\perp\otimes\bigotimes_{i=2}^{k-1}S^{\omega_i}V_i\right)\oplus\left(\xx_1^{\omega_1-1}\langle\xx_1\rangle\otimes\bigotimes_{i=2}^{k-1}S^{\omega_i}V_i\right)=\xx_1^{\omega_1-1}\C^{n_1}\otimes\bigotimes_{i=2}^{k-1}S^{\omega_i}V_i\,.
\]
On one hand, the left-hand side of \eqref{eq: identity k-1 factors} can be rewritten as
\[
\xx_1^{\omega_1-1}\cdot\ww_1\otimes\xx_2^{\omega_2}\otimes\cdots\otimes\xx_{k-1}^{\omega_{k-1}}+\sum_{i=2}^{k-1}\xx_1^{\omega_1}\otimes\cdots\otimes\xx_i^{\omega_i-1}\cdot\ww_i\otimes\cdots\otimes\xx_{k-1}^{\omega_{k-1}}\,.
\]
The first summand is a nonzero element of $\xx_1^{\omega_1-1}\langle\xx_1\rangle^\perp\otimes\bigotimes_{i=2}^{k-1}S^{\omega_i}V_i$, whereas the second summand is a nonzero element of $\xx_1^{\omega_1-1}\langle\xx_1\rangle\otimes\bigotimes_{i=2}^{k-1}S^{\omega_i}V_i$.
On the other hand, the right-hand side of \eqref{eq: identity k-1 factors} lives only in $\xx_1^{\omega_1-1}\langle\xx_1\rangle\otimes\bigotimes_{i=2}^{k-1}S^{\omega_i}V_i$. Hence $F_1\neq F_2$ and $(W_1+\cdots+W_{k-1})\cap W_k=\{0\}$. The identity \eqref{eq: dim fiber q r isotropic factors} again follows applying the induction step.

We conclude that, if $(W_1+\cdots+W_{k-1})\cap W_k\neq\{0\}$, then $q_i(\xx_i)=0$ for all $i\in[k]$.
Under this assumption, we have that $\dim\left[(W_1+\cdots+W_{k-1})\cap W_k\right]=1$: indeed, the only possibility to have $F_1=F_2$ is that $\ww_i\in\langle\xx_i\rangle$ for all $i\in[k]$. Hence, by induction,
\[
\dim(W_1+\cdots+W_k) = \dim(W_1+\cdots+W_{k-1})+\dim(W_k)-\dim\left[(W_1+\cdots+W_{k-1})\cap W_k\right]=\sum_{i=1}^{k}(n_i-1)-(k-1)\,,
\]
which agrees with \eqref{eq: dim fiber q r isotropic factors} when $r=k$.
\end{proof}

\begin{Definition}[{\bf Kalman strata}]\label{def: kalmanstrata}
For every $i\in[k]$, let $Z_i\subset \PP( V_i)$ be a projective variety and consider the product $Z = \prod_{i=1}^k Z_i$. Given a subset $J\subset[k]$, we define the product
\[
Q_J\coloneqq \prod_{j\in J}Q_j\times\prod_{j\notin J}[\PP(V_j)\setminus Q_j]\subset\PP\,.
\]
The {\it partially isotropic Kalman strata} with respect to $Z$ are
\[
\kappa^{J}_{\nn, \bomega}(Z)\coloneqq\overline{\{ T\in \PP(S^{\bomega}V) \mid\mbox{$T$ has a singular vector $k$-tuple $([\xx_1], \ldots, [\xx_k])\in Z\cap Q_J$}\}}\,.
\]
In particular,
\begin{itemize}
\item for $J=\emptyset$, we call $\kappa^{nor}_{\nn, \bomega}(Z)\coloneqq\kappa^{\emptyset}_{\nn, \bomega}(Z)$ the {\it normalized Kalman variety} with respect to $Z$ because we may assume that all components $\xx_i$ are normalized with respect to the inner product $q_i$. We also denote by $Q_{nor}$ the product $Q_{\emptyset}=\prod_{j=1}^k[\PP(V_j)\setminus Q_j]$. 
\item for $J=[k]$, we call $\kappa^{iso}_{\nn, \bomega}(Z)\coloneqq\kappa^{[k]}_{\nn, \bomega}(Z)$ the {\it totally isotropic Kalman variety} with respect to $Z$ because all components $\xx_i$ are isotropic. In this case
\[
\kappa^{iso}_{\nn, \bomega}(Z)=\{ T\in \PP(S^{\bomega}V) \mid\mbox{$T$ has a singular vector $k$-tuple $([\xx_1], \ldots, [\xx_k])\in Z\cap Q$}\}\,,
\]
namely the right-hand side is already closed. Here we use the shorthand $Q\coloneqq Q_{[k]}=\prod_{i=1}^k Q_i$. If $Z=Q$, we indicate this variety simply with $\kappa^{iso}_{\nn, \bomega}$.
\end{itemize}
Note that all the loci $\kappa^{J}_{\nn, \bomega}(Z)$ are closed by definition, and they may be reducible. We generally refer to all of them (to their irreducible components) as {\it Kalman strata}.
If $\bomega={\bf 1}$, we use the shorthand $\kappa^{J}_{\nn}(Z)=\kappa^{J}_{\nn,{\bf 1}}(Z)$.
\end{Definition}

\begin{Definition}\label{def: genKalman}
For every $i\in[k]$, let $Z_i\subset \PP( V_i)$ be a projective variety.
The {\it generalized Kalman variety} with respect to $Z=\prod_{i=1}^k Z_i$ is
\begin{equation}\label{eq: def genKalman}
\kappa_{\nn, \bomega}(Z)\coloneqq\{T\in \PP(S^{\bomega}V) \mid\mbox{$T$ has a singular vector $k$-tuple $([\xx_1], \ldots, [\xx_k])\in Z$}\}\,.
\end{equation}
The Kalman strata $\kappa_{\nn, \bomega}^J(Z)$ of Definition \ref{def: kalmanstrata} are closed subvarieties of the generalized Kalman variety. Moreover, if $Z=Q=\prod_{i=1}^kQ_i$, then $\kappa_{\nn, \bomega}(Q)=\kappa^{iso}_{\nn, \bomega}$. 
\end{Definition}

\noindent Similarly as in Remark \ref{rmk: spectral is Zar closed}, the right-hand side of \eqref{eq: def genKalman} is closed in the Zariski topology.

\begin{Definition}
Consider the spectral variety $\Sigma_{\nn,\bomega}\subset \Pi_{\nn,\bomega}$.
For every $i\in[k]$, let $Z_i\subset \PP(V_i)$ be a projective variety and consider the product $Z= \prod_{i=1}^k Z_i\subset\PP$.
The {\em spectral variety of type $(\nn,\bomega)$ restricted to $Z$} is the incidence variety $\Sigma_{\nn,\bomega}(Z) \coloneqq \Sigma_{\nn,\bomega}\cap[\PP(S^{\bomega}V)\times Z]$. In particular,
\[
\Sigma_{\nn,\bomega}(Z) = \{ (T, [\xx_1],\ldots, [\xx_k])\mid \text{$([\xx_1],\ldots,[\xx_k])\in Z$ is a singular vector $k$-tuple for $T$}\}\,.
\]
\end{Definition}

Similarly as in \eqref{eq: diagram spectral variety}, in the following proofs we will consider the diagram of projections

\begin{equation}\label{eq: diagram spectral variety Z}
\begin{tikzcd}
& \Sigma_{\nn,\bomega}(Z) \arrow{dl}[swap]{\alpha_Z} \arrow{dr}{\beta_Z} & \\
\PP(S^{\bomega}V) & & Z\,.
\end{tikzcd}
\end{equation}

Before we proceed, we recall a standard lemma needed in the next proofs. The (omitted) proof relies on a direct application of \cite[Proposition 7.10]{Hart} and \cite[Exercise 11.4.C]{Vakil}. 

\begin{Lemma}\label{lem: components}
Let $\varphi \colon M\to N$ be a surjective morphism of projective varieties and assume that $N$ is connected.
Suppose there exists a finite collection $\{S_{i}\}_{i\in I}$ of irreducible quasi-projective subvarieties of $N$ such that $N = \bigcup_{i\in I} S_{i}$ with the property that, for each
$i\in I$, the restriction $\varphi_{|\varphi^{-1}(S_{i})}\colon \varphi^{-1}(S_{i})\to S_{i}$ has equidimensional linear projective fibers. Then:
\begin{enumerate}
\item[$(i)$] Each closed subvariety $\overline{\varphi ^{-1}(S_{i})}\subset M$ is irreducible.
\item[$(ii)$] $M$ is connected and $M = \bigcup_{i\in I}\overline{\varphi^{-1}(S_{i})}$. In particular, the irreducible components of $M$ have the form $\overline{\varphi^{-1}(S_{i})}$ for some $i\in I$.
\item[$(iii)$] One has $\overline{\varphi^{-1}(S_{i})}\setminus\varphi^{-1}(S_{i})\subset\bigcup_{j\in I\setminus\{i\}}\varphi^{-1}(S_{j})$.
\end{enumerate}
\end{Lemma}

\begin{Proposition}\label{prop: DimensionsKalmanstrata}
For every $i\in[k]$, let $Z_i\subset \PP(V_i)$ be an irreducible projective variety of codimension $\delta_i$ and consider the product $Z= \prod_{i=1}^k Z_i\subset\PP$.
We assume that $Z_i\not\subset Q_i$ for all $i\in[k]$ and let $\delta\coloneqq\sum_{i=1}^k\delta_i$. Then: 
\begin{enumerate}
\item[$(i)$] The variety $\kappa^{nor}_{\nn, \bomega}(Z)$ is irreducible of dimension $\dim(\kappa^{nor}_{\nn, \bomega}(Z)) = \dim(\PP(S^{\bomega}V))-\delta$. 
\item[$(ii)$] For any $J\neq\emptyset$, we have $\dim(\kappa^{J}_{\nn, \bomega}(Z))<\dim(\kappa^{nor}_{\nn, \bomega}(Z))$. 
\end{enumerate}
\end{Proposition}
\begin{proof}

$(i)$. Consider the restricted spectral variety $\Sigma_{\nn,\bomega}(Z)$ along with the two projections $\alpha_Z$ and $\beta_Z$ introduced in diagram \eqref{eq: diagram spectral variety Z}.
The morphism $\beta_Z$ is surjective and the fiber $\beta_Z^{-1}([\xx_1],\ldots, [\xx_k])$ is set-theoretically the collection of tensors $T\in \PP(S^{\bomega}V)$ that possess $([\xx_1],\ldots, [\xx_k])\in Z$ as a singular vector $k$-tuple.

\noindent Let $Q=\prod_{i=1}^k Q_i\subset\PP$ and $Y = \overline{\beta_Z^{-1}(Z\setminus Q)}\subset \Sigma_{\nn,\bomega}(Z)$. Now consider the morphism
\[
\psi\coloneqq {\beta_Z}_{|\beta_Z^{-1}(Z\setminus Q)}\colon \beta_Z^{-1}(Z\setminus Q)\to Z\setminus Q\,.
\]
By Theorem \ref{thm: fibers}, $\psi$ is a surjective morphism with equidimensional linear projective fibers of codimension $\sum_{i=1}^k(n_i-1)$ in $\PP(S^{\bomega}V)$. 
By Lemma \ref{lem: components}$(i)$, $Y$ is irreducible.

Now consider the first projection $\alpha_Z\colon\Sigma_{\nn,\bomega}(Z)\to \PP(S^{\bomega}V)$. The morphism $\alpha_Z$ is projective and hence closed. One has $\kappa^{nor}_{\nn, \bomega}(Z) = \alpha_Z(Y)$. To see this, notice that by definition we have:
\[
\alpha_Z(\beta_Z^{-1}(Z\setminus Q)) = \{ T\in \PP(S^{\bomega}V) \mid\mbox{$T$ has a singular vector $k$-tuple $([\xx_1], \ldots, [\xx_k])\in Z\setminus Q$}\}\,.
\]
One finds $\kappa^{nor}_{\nn, \bomega}(Z)=\overline{\alpha_Z(\beta_Z^{-1}(Z\setminus Q))}=\overline{\alpha_Z(\overline{\beta_Z^{-1}(Z\setminus Q)})}=\overline{\alpha_Z(Y)}=\alpha_Z(Y)$. Here the first equality follows by Definition \ref{def: kalmanstrata}, the second by continuity of $\alpha_Z$,  and the last follows from $\alpha_Z$ being closed. Moreover, since $Y$ is irreducible, then so is $\kappa^{nor}_{\nn, \bomega}(Z)$.

In order to calculate the codimension of $\kappa^{nor}_{\nn, \bomega}(Z)$, note that $\dim(Y)=\dim(\PP(S^{\bomega}V))+\dim(Z)-\sum_{i=1}^k(n_i-1)=\dim(\PP(S^{\bomega}V))-\delta$. 
Observe that the general fiber of the restriction of $\alpha_Z$ to $Y$ is finite. Indeed, a point of the product $Z\setminus Q$ is the equivalence class of a tuple of non-isotropic vectors. The group $\mathrm{SO}(V_1)\times\cdots\times\mathrm{SO}(V_k)$ acts transitively on equivalence classes of non-isotropic vectors. Thus a general tensor in the fiber under $\beta_Z$ of a point in $Z\setminus Q$ has a finite number of singular vector $k$-tuples. This implies the finiteness of the restriction of $\alpha_Z$ to $Y$. Then $\dim(\kappa^{nor}_{\nn, \bomega}(Z)) = \dim(Y)$. So $\mathrm{codim}(\kappa^{nor}_{\nn, \bomega}(Z)) = \delta$.

\medskip
\noindent $(ii)$. Fix a nonempty subset $J\subset[k]$ and consider the quasi-projective subvariety $Z\cap Q_J$ of $Z$. 
By Theorem \ref{thm: fibers}, the fiber of $\beta_Z$ at every point of $Z\cap Q_J$ has dimension $\dim(\PP(S^{\bomega}V))-\sum_{i=1}^k(n_i-1)+(|J|-1)$.
Let $Y_J = \overline{\beta^{-1}(Z\cap Q_J)}$. Thus 
\[
\dim(Y_J) = \sum_{j\notin J} \dim(Z_j) + \sum_{j\in J}[\dim(Z_j)-1] + \dim(\PP(S^{\bomega}V))-\sum_{i=1}^k(n_i-1)+(|J|-1) = \dim(\PP(S^{\bomega}V))-\delta - 1\,.
\]
As before, we find that $\kappa^{J}_{\nn, \bomega}(Z) = \alpha_Z(Y_J)$.  Thus the dimension of $\kappa^{J}_{\nn, \bomega}(Z)$ is at most 
that of $Y_J$, which finishes the proof. 
\end{proof}

\noindent The last step before proving Theorem \ref{thm: OSh, partially symmetric case} is establishing the irreducibility of the spectral variety $\Sigma_{\nn,\bomega}$.

\begin{Proposition}\label{prop: irredsigmax}
The spectral variety $\Sigma_{\nn,\bomega}\subset \Pi_{\nn,\bomega}$ is irreducible and $\Sigma_{\nn,\bomega} = \overline{\beta^{-1}(\PP\setminus Q)}$. 
\end{Proposition}
\begin{proof}
The product $\PP = \prod_{i=1}^k \PP(V_i)$ comes equipped with the $i$-th projection map $\pi_i\colon\PP\to\PP(V_i)$. On the product $\PP$, consider the pull-back vector bundles
\[
\varepsilon_i\coloneqq \pi_i^*\mathcal{Q}_{\PP(V_i)}(\omega_1,\ldots,\omega_{i-1},\omega_i-1,\omega_{i+1},\ldots,\omega_k)\quad\forall\,i\in[k]\,.
\]
The fiber of each vector bundle $\varepsilon_i$ at $[\xx]=[\xx_1^{\omega_1}\otimes\cdots\otimes \xx_k^{\omega_k}]$ is isomorphic to
\[
\mathrm{Hom}(\xx_1^{\omega_1}\otimes\cdots\otimes \xx_i^{\omega_i-1}\otimes\cdots\otimes \xx_k^{\omega_k}, V_i/\langle \xx_i\rangle)\,.
\]
Let $\bomega_i=(\omega_1,\ldots,\omega_i-1,\ldots,\omega_k)$. Every tensor $T\in S^{\bomega}V$ induces a section $s_T$ of $\varepsilon_i$ which corresponds to the composition
\[
\langle \xx_1^{\omega_1}\otimes\cdots\otimes \xx_i^{\omega_i-1}\otimes\cdots\otimes \xx_k^{\omega_k}\rangle\stackrel{i}{\hookrightarrow}S^{\bomega_i}V\stackrel{T}{\longrightarrow}V_i\stackrel{\pi}{\twoheadrightarrow}\frac{V_i}{\langle \xx_i\rangle}.
\]
This section vanishes on $[\xx]$ if and only if
\[
\mathrm{rank}
\begin{pmatrix}
T(\xx_1^{\omega_1}\otimes\cdots\otimes \xx_i^{\omega_i-1}\otimes\cdots\otimes \xx_k^{\omega_k})\\
\xx_i
\end{pmatrix}
\le 1\,.
\]
Define the vector bundle $\varepsilon\coloneqq\bigoplus_{i=1}^k\varepsilon_i$ on $\PP$.
For every $T\in S^{\bomega}V$, the diagonal section
$(s_T,\ldots,s_T)\in H^0(\varepsilon)$ vanishes on $\xx$ if and only if $(\xx_1,\ldots,\xx_k)$ is a partially symmetric singular $k$-tuple of $T$.

\noindent Using the vector bundle $\varepsilon$ and the natural projections $\tilde{\alpha}\colon \Pi_{\nn,\bomega}\to \PP(S^{\bomega}V)$ and $\tilde{\beta}\colon \Pi_{\nn,\bomega}\to\PP$, we introduce the {\it Friedland-Ottaviani vector bundle} on $\Pi_{\nn,\bomega}$: 
\[
E\coloneqq \tilde{\alpha}^*(\mathcal{O}_{\PP(S^{\bomega}V)}(1))\otimes \tilde{\beta}^*(\varepsilon)\,.
\]
The vector bundle $E$ has rank $\sum_{i}(n_i-1)$  and was first introduced by Friedland and Ottaviani for studying singular vector $k$-tuples; see \cite[Theorem 12]{FO}. The vanishing locus $Z(s)$ of the section $s\in H^0(E)$ given by the map $T\mapsto (s_T, \ldots, s_T)$ equals the spectral variety $\Sigma_{\nn,\bomega}\subset \Pi_{\nn,\bomega}$ \cite[Lemma 9]{FO}. 

The variety $\overline{\beta^{-1}(\PP\setminus Q)}$ is an irreducible component of $\Sigma_{\nn,\bomega}$ by Lemma \ref{lem: components}$(i)$.
By Theorem \ref{thm: fibers} and Lemma \ref{lem: components}$(ii)$, if some another irreducible component of $\Sigma_{\nn,\bomega}$ exists, its codimension must be strictly higher than $\codim(\overline{\beta^{-1}(\PP\setminus Q)}) = \sum_{i}(n_i-1) = \mathrm{rank}(E)$.
In any affine open $U\subset \Pi_{\nn,\bomega}$ trivializing the vector bundle $E$, $Z(s)|_{U} = \Sigma_{\nn,\bomega}|_{U}$ is defined by the vanishing of $\mathrm{rank}(E)$ equations.
Therefore, by Krull's principal ideal theorem, on the affine chart $U\subset X$ there cannot be any irreducible component of $\Sigma_{\nn,\bomega}|_{U}$ whose codimension is strictly higher than $\mathrm{rank}(E)$.
Gluing an affine open cover of $\Pi_{\nn,\bomega}$ trivializing $E$, then one finds that there cannot be other irreducible components of $\Sigma_{\nn,\bomega}$ besides $\overline{\beta^{-1}(\PP\setminus Q)}$. This proves the equality and irreducibility.  
\end{proof}

\noindent We are ready to prove Theorem \ref{thm: OSh, partially symmetric case}. 

\begin{proof}[Proof of Theorem \ref{thm: OSh, partially symmetric case}]
Consider the product $\Pi_{\nn,\bomega}$ and the spectral variety $\Sigma_{\nn,\bomega}(Z)$ restricted to $Z$.
Moreover, we consider the projections $\beta_Z$ and $\alpha_Z$ of diagram \eqref{eq: diagram spectral variety Z}.
By Theorem \ref{thm: fibers} and Lemma \ref{lem: components}$(ii)$, the irreducible components of $\Sigma_{\nn,\bomega}(Z)$ are, up to permuting factors, of the form $Y'_J = \overline{\beta_Z^{-1}(Z'_J)}$, where $Z'_J$ is an irreducible component of $Z\cap Q_J$ for some $J\subset[k]$, as $Z\cap Q_J$ might be reducible.
Note that the generalized Kalman variety is $\kappa_{\nn, \bomega}(Z) = \alpha_Z(\Sigma_{\nn,\bomega}(Z))$.
Therefore the irreducible components of the generalized Kalman variety are of the form $\alpha_Z(Y'_J)$, each one of those is a Kalman stratum. 

From the dimensions calculated in Proposition \ref{prop: DimensionsKalmanstrata}, it follows that the unique highest-dimensional irreducible component of $\kappa_{\nn, \bomega}(Z)$ is the normalized Kalman variety $\kappa^{nor}_{\nn, \bomega}(Z)$. Therefore codimension and degree of the generalized Kalman variety coincide with the ones of $\kappa^{nor}_{\nn, \bomega}(Z)$. 

\noindent Again by Lemma \ref{lem: components}$(ii)$, $\Sigma_{\nn,\bomega}(Z)$ is connected. Since $\alpha_Z$ is continuous, $\kappa_{\nn, \bomega}(Z)$ is connected.

We are left with the computation of $\deg(\kappa_{\nn, \bomega}(Z))=\deg(\kappa^{nor}_{\nn, \bomega}(Z))$. As in the proof of Proposition \ref{prop: irredsigmax}, we work with the Friedland-Ottaviani vector bundle $E$ of rank $\sum_{i}(n_i-1)$ on $\Pi_{\nn,\bomega}$. 

Denote by $D=\prod_{i=1}^k\binom{n_i+\omega_i-1}{\omega_i}-1$ the dimension of $\PP(S^{\bomega}V)$.
We recall that $\alpha$ and $\beta$ appearing in diagram \eqref{eq: diagram spectral variety} are the restrictions of $\tilde{\alpha}$ and $\tilde{\beta}$ to $\Sigma_{\nn,\bomega}$. By Proposition \ref{prop: irredsigmax}, $\Sigma_{\nn,\bomega} = Z(s) = \overline{\beta^{-1}(\PP\setminus Q)}$ and has codimension $\sum_{i}(n_i-1)$ in $\Pi_{\nn,\bomega}$. 
Therefore the top Chern class $c_{\sum_i(n_i-1)}(E)$ of $E$ satisfies $c_{\sum_i(n_i-1)}(E) = [Z(s)]\in A^{*}(\Pi_{\nn,\bomega})$, where $A^{*}(\Pi_{\nn,\bomega}) =\C[h,t_1,\ldots,t_k]/(h^{D+1},t_1^{n_1},\ldots,t_k^{n_k})$ ($h=c_1(\mathcal{O}_{\PP(S^{\bomega}V)}(1))$ and $t_i=c_1(\mathcal{O}_{\PP(V_i)}(1))$) is the Chow ring of $\Pi_{\nn,\bomega}$. 

We now look at the intersection $Z(s)\cap \beta^{-1}(Z)\subset \Pi_{\nn,\bomega}$, which may be reducible. Note that this intersection 
contains $\overline{\beta^{-1}(Z\setminus Q)}$, which has codimension $\sum_i(n_i-1)+\mathrm{codim}(Z)=\sum_i(n_i-1)+\delta$. We now check whether the intersection  contains other components whose codimension is smaller than or equal to the codimension of  $\overline{\beta^{-1}(Z\setminus Q)}$. By Lemma \ref{lem: components}$(iii)$, $Z(s)\setminus \beta^{-1}(\PP\setminus Q)$ is a union of quasi-projective varieties of the form $\beta^{-1}(\prod_{j\in J}Q_j\times\prod_{j\notin J}[\PP(V_j)\setminus Q_j])$.
Therefore $(Z(s)\setminus \beta^{-1}(\PP\setminus Q))\cap \beta^{-1}(Z)$ is a union of quasi-projective varieties of the form $\beta^{-1}(\prod_{j\in J}Q_j\cap Z_j\times\prod_{j\notin J}[Z_j\setminus Q_j])$. As in the proof of Proposition \ref{prop: DimensionsKalmanstrata}, all these quasi-projective 
varieties have dimensions strictly smaller than the dimension of $\overline{\beta^{-1}(Z\setminus Q)}$. 
In conclusion, the intersection $Z(s)\cap \beta^{-1}(Z)\subset \Pi_{\nn,\bomega}$ contains a unique irreducible component in codimension $\sum_i(n_i-1)+\delta$; the other potential components are in strictly higher codimensions. 

The previous paragraph shows that the element  $c_{\sum_i(n_i-1)}(E)\prod_{i=1}^k\deg(Z_i)\cdot \beta^*(\pi_i^*(\mathcal{O}_{Z_i}(1)^{\delta_i}))$ of the Chow ring is nonzero in codimension $\sum_i(n_i-1)+\delta$ and zero in smaller codimension. Thus the intersection number 
\[
\alpha^*(c_1(\mathcal{O}_{\PP(S^{\bomega}V)}(1))^{D-\delta}\cdot c_{\sum_i(n_i-1)}(E)\prod_{i=1}^k\deg(Z_i)\cdot \beta^*(\pi_i^*(\mathcal{O}_{Z_i}(1)^{\delta_i}))
\]
in the Chow ring $A^{*}(\Pi_{\nn,\bomega})$ is the degree of $\alpha(\overline{\beta^{-1}(Z\setminus Q)}) = \kappa^{nor}_{\nn, \bomega}(Z)$. 

The Friedland-Ottaviani vector bundle $E$ is the direct sum of $k$ summands. By Euler's exact sequence, each of these summands has Chern polynomial
\[
\frac{(1+\hat{t_i}+h)^{n_i}}{1+\widehat{t_i}+h-t_i}=\frac{(1+\hat{t_i}+h)^{n_i-1}}{1-\frac{t_i}{1+\widehat{t_i}+h}}=\sum_{j=0}^\infty(1+\hat{t_i}+h)^{n_i-1-j}t_i^j=\sum_{j=0}^{n_i-1}(1+\hat{t_i}+h)^{n_i-1-j}t_i^j\,,
\]
where $\widehat{t_i}\coloneqq\left(\sum_{j=1}^{k}\omega_j t_j\right)- t_i$. The statement then follows as in the proof of \cite[Theorem 3.4]{OSh},
noting that we carry as extra factor the product of the degrees $\deg(Z_i)$. 
\end{proof}

\begin{Remark}
We do not know whether the generalized Kalman variety $\kappa_{\nn, \bomega}(Z)$, under the assumptions of Theorem \ref{thm: OSh, partially symmetric case}, is always irreducible. The issue is that we do not know in general how the Kalman strata intersect. 
\end{Remark}

The next examples suggest that these varieties might be subtle (even for matrices) and show irreducibility is not true when the assumptions $Z_i\not\subset Q_i$ of Theorem \ref{thm: OSh, partially symmetric case} are weakened. 

\begin{Example}
Let $k=2, \nn = (2,3)$, and $\bomega=(1,1)$. 
We assume that $\PP(V_1)\cong\PP_{\mC}^1$, $\PP(V_2)\cong\PP_{\mC}^2$, $Q_1=\mathcal{V}(x_{1,1}^2+x_{1,2}^2)\subset\PP_{\mC}^1$ and $Q_2=\mathcal{V}(x_{2,1}^2+x_{2,2}^2+x_{2,3}^2)\subset\PP_{\mC}^2$.
On one hand, the product $\PP_{\mC}^1\times Q_2$ is irreducible in $\PP_{\mC}^1\times\PP_{\mC}^2$. 
On the other hand, the generalized Kalman variety $\kappa_{\nn}(\PP_{\mC}^1\times Q_2)$ is a degree $8$ hypersurface with three irreducible components: 
\begin{enumerate}
\item[$(i)$] the totally isotropic Kalman variety $\kappa^{iso}_{\nn}$, that is a degree $4$ hypersurface with two irreducible components. Using coordinates $a_{ij}$ for the space $\PP(\C^2\otimes\C^3)\cong\PP_{\mC}^5$, its equation is
\[
4(a_{11}a_{21}+a_{12}a_{22}+a_{13}a_{23})^2+(a_{11}^{2}+a_{12}^{2}+a_{13}^{2}-a_{21}^{2}-a_{22}^{2}-a_{23}^{2})^2\,.
\]
The first summand corresponds to the Euclidean inner product of the two rows of $A=(a_{ij})\in\C^2\otimes\C^3$.
The second summand is the difference between the Euclidean norms of the two rows of $A$.
\item[$(ii)$] the Kalman stratum $\kappa^{\{2\}}_{\nn}(\PP_{\mC}^1\times Q_2)$, that is a degree $4$ irreducible hypersurface of equation
\[
(a_{11}a_{22}-a_{12}a_{21})^2+(a_{11}a_{23}-a_{13}a_{21})^2+(a_{12}a_{23}-a_{13}a_{22})^2\,,
\]
namely the sum of squares of the three maximal minors of $A$.
\end{enumerate}
If $\nn = (2,4)$, then the generalized Kalman variety $\kappa_{\nn}(\PP_{\mC}^1\times Q_2)$, where $Q_2=\mathcal{V}(x_{2,1}^2+\cdots+x_{2,4}^2)\subset\PP_{\mC}^3$, is a degree $8$ hypersurface in $\PP(\C^2\otimes\C^4)\cong\PP_{\mC}^7$ with three analogous irreducible components. 

\noindent For $\nn = (3,3)$, one finds 
$\kappa_{\nn}(\PP_{\mC}^2\times Q_2) = \kappa^{iso}_{\nn} = \kappa_{\nn}(Q_1\times\PP_{\mC}^2)$.
\end{Example}

\begin{Example}
Let us consider the matrix case $k=2$, $\nn=(4,4)$, $\bdelta=(2,1)$ and $\bomega=(1,1)$. Consider the twisted cubic curve $Z_1\subset \PP_{\mC}^3$ and a general quadric surface $Z_2\subset\PP(V_2)\cong\PP_{\mC}^3$. Applying Theorem \ref{thm: OSh, partially symmetric case}, we have
\[
\deg(\kappa^{nor}_{\nn}(Z)) = \deg(Z_1)\deg(Z_2)d((4,4),(2,1),(1,1))\,, 
\]
where $d((4,4),(2,1),(1,1))$ is the coefficient of the monomial $h^3t_1t_2^2$ in the polynomial
\[
\sum_{i,j=0}^3(t_1+h)^{3-j}(t_2+h)^{3-1}t_1^it_2^j=\cdots+20\,h^3t_1t_2^2+\cdots\,,
\]
therefore $\deg(\kappa^{nor}_{\nn}(Z)) = 3 \cdot 2\cdot 20=120$.
We provide a \verb+Macaulay2+ code \cite{GS} to verify symbolically this degree computation.
We speed up our degree computation by working over the finite field $\mathbb Z/\mathbb Z_{101}$ and restricting to a $3$-dimensional subspace $H\subset\PP(\C^4\otimes\C^4)$, since $\kappa^{nor}_{\nn}(Z)$ has codimension $2+1=3$:

{\small
\begin{Verbatim}[commandchars=\\\{\}]
K = \textcolor{purple}{ZZ}/101; R = K[u_0..u_3,x_(1,1)..x_(2,4)];
coeffs = \textcolor{ForestGreen}{toList}((1,1)..(4,4)); \textcolor{blue}{for} f \textcolor{blue}{in} coeffs \textcolor{blue}{do} a_f = \textcolor{ForestGreen}{sum}(m+1, i-> \textcolor{ForestGreen}{random}(K)*u_i);
M = \textcolor{ForestGreen}{sum}(4, i-> \textcolor{ForestGreen}{sum}(4, j-> a_(i+1,j+1)*x_(1,i+1)*x_(2,j+1)));
xx1 = \textcolor{ForestGreen}{matrix}\{\{x_(1,1)..x_(1,4)\}\}; xx2 = \textcolor{ForestGreen}{matrix}\{\{x_(2,1)..x_(2,4)\}\};
I = \textcolor{ForestGreen}{minors}(2,\textcolor{ForestGreen}{contract}(xx1,M)||xx1)+\textcolor{ForestGreen}{minors}(2,\textcolor{ForestGreen}{contract}(xx2,M)||xx2)+
    \textcolor{ForestGreen}{ideal}(\textcolor{ForestGreen}{sum}(4, i-> x_(1,i+1)^2)-1,\textcolor{ForestGreen}{sum}(4, i-> x_(2,i+1)^2)-1);
\end{Verbatim}
}
Here \verb+I+ is the ideal of relations among the singular vector pairs in the symbolic coordinates of \verb+M+. We impose the existence of a singular vector pair on $Z=Z_1\times Z_2$ as follows:
{\small
\begin{Verbatim}[commandchars=\\\{\}]
Z_1 = \textcolor{ForestGreen}{minors}(2, \textcolor{ForestGreen}{matrix}\{\{x_(1,1)..x_(1,3)\},\{x_(1,2)..x_(1,4)\}\});
Z_2 = \textcolor{ForestGreen}{ideal}((\textcolor{ForestGreen}{symmetricPower}(2,xx2)*\textcolor{ForestGreen}{random}(K^10,K^1))_(0,0));
J = I + Z_1 + Z_2;
\end{Verbatim}
}
Finally, the ideal of $\kappa^{nor}_{\nn}(Z)$ is computed with 
{\small
\begin{Verbatim}[commandchars=\\\{\}]
Kalman = \textcolor{ForestGreen}{eliminate}(\textcolor{ForestGreen}{toList}(x_(1,1)..x_(2,4)), J);
\textcolor{ForestGreen}{degree} Kalman  \textcolor{orange}{-- this confirms that the variety has degree 120}
\end{Verbatim}
}
\end{Example}

\begin{Remark}\label{rmk: degree Kalman symmetric case}
Note that Theorem \ref{thm: OSh, partially symmetric case} includes the symmetric case. Consider an irreducible variety $Z\subset \PP(V)$ not contained in the isotropic quadric $Q\subset\PP(V)$ and the generalized Kalman variety $\kappa_{n, \omega}(Z)\coloneqq\kappa_{\nn, \bomega}(Z)$, where $\bdelta=(\delta)$, $\delta=\codim(Z)$, $\nn = (n)$ and $\bomega = (\omega)$. Its degree is $d(n,\delta,\omega)\deg(Z)$, where $d(n,\delta,\omega)$ is the coefficient of the monomial $h^\delta t^{n-\delta-1}$ in the polynomial
\[
\frac{[(\omega-1)t+h]^n-t^n}{(\omega-1)t+h-t} = \sum_{i=0}^{n-1}[(\omega-1)t+h]^{n-1-i}t^i = \sum_{i=0}^{n-1}\sum_{j=0}^{n-1-i}\binom{n-1-i}{j}(\omega-1)^{n-1-i-j}t^{n-1-j}h^j\,.
\]
The coefficient of the monomial $h^\delta t^{n-\delta-1}$ in the last polynomial is
\[
\sum_{i=0}^{n-1}\binom{n-1-i}{\delta}(\omega-1)^{n-\delta-1-i}=\sum_{i=0}^{n-\delta-1}\binom{n-1-i}{\delta}(\omega-1)^{n-\delta-1-i}=\sum_{j=0}^{n-\delta-1}\binom{\delta+j}{j}(\omega-1)^j\,,
\]
which recovers \cite[Theorem 2.3]{OSh}.
\end{Remark}

\section{The totally isotropic Kalman variety}\label{sec: iso}
In Theorem \ref{thm: OSh, partially symmetric case} we assumed that each $Z_i$ is not contained in the corresponding isotropic quadric $Q_i$. 
On the opposite side, we exhibit a description of totally isotropic Kalman varieties as dual varieties. Before we prove the precise statement, 
we need the following lemma. 

\begin{Lemma}\label{lem: when Segre-Veronese embedding isotropic quadrics is hypersurface}
Consider the product $Q=\prod_{i=1}^k Q_i$ of isotropic quadrics $Q_i\subset\PP(V_i)$.
Let $v_{\nn,\bomega}\colon\PP\to\PP(S^\bomega V)$ be the degree-$\bomega$ Segre-Veronese embedding of $\PP=\prod_{i=1}^k\PP(V_i)$. Then the dual variety $[v_{\nn,\bomega}(Q)]^\vee$ is a hypersurface for every choice of $\nn$ and $\bomega$.
\end{Lemma}

\begin{proof}
Considering a slight modification of \cite[Chapter 1, Corollary 5.10]{GKZ}, the variety $[v_{\nn,\bomega}(Q)]^\vee$ is a hypersurface if and only if
\begin{align}\label{eq: system inequalities GKZ}
\dim(v_{n_j,\omega_j}(Q_j))+\mathrm{codim}([v_{n_j,\omega_j}(Q_j)]^\vee)-1 \le \dim(v_{\nn,\bomega}(Q))\quad\forall j\in[k]\,,
\end{align}
where each variety $v_{n_j,\omega_j}(Q_j)$, which is the degree-$\omega_j$ Veronese embedding of $Q_j$. The variety $[v_{n_j,\omega_j}(Q_j)]^\vee$ is a hypersurface for all $\omega_j\ge 1$ by \cite[Corollary 4.8]{Sod}. Therefore, for all $j\in[k]$ the corresponding inequality in \eqref{eq: system inequalities GKZ} becomes $n_j-2\le\dim(v_{\nn,\bomega}(Q))=n_1+\cdots+n_k-2k$, which is clearly satisfied.
\end{proof}

\begin{Theorem}\label{thm: codim degree Kalman partially symmetric isotropic}
Assume $n_j\ge 3$ for all $j\in[k]$. Consider the product $Q=\prod_{i=1}^k Q_i$ of isotropic quadrics $Q_i\subset\PP(V_i)$. Let $\kappa^{iso}_{\nn, \bomega}$ be the totally isotropic Kalman variety, i.e., the variety of partially symmetric tensors having a singular vector $k$-tuple in $Q$. Let $v_{\nn,\bomega}\colon\PP\to\PP(S^\bomega V)$ be the degree-$\bomega$ Segre-Veronese embedding of $\PP=\prod_{i=1}^k\PP(V_i)$. Then
\begin{equation}\label{eq: Kalman dual Segre-Veronese}
\kappa^{iso}_{\nn, \bomega}=[v_{\nn,\bomega}(Q)]^\vee\,.
\end{equation}
In particular, $\kappa^{iso}_{\nn, \bomega}$ is an irreducible hypersurface of $\PP(S^\bomega V)$. Its degree is equal to the degree of $[v_{\nn,\bomega}(Q)]^\vee$, that is
\begin{equation}\label{eq: degree Kalman partially symmetric isotropic}
2^k\sum_{j=0}^N(-1)^j(N+1-j)!\sum_{|\alpha|=j}\left[\prod_{l=1}^k\frac{\omega_l^{n_l-2-\alpha_l}}{(n_l-2-\alpha_l)!}\sum_{\substack{\beta_l=0\\l\in[k]}}^{\alpha_l}\prod_{l=1}^k\binom{n_l}{\beta_l}(-2)^{\alpha_l-\beta_l}\right]
\end{equation}
where $N=\dim(v_{\nn,\bomega}(Q))=n_1+\cdots+n_k-2k$.

If $J\coloneqq\{j\in[k]\mid n_j=2\}\neq\emptyset$, then identity \eqref{eq: Kalman dual Segre-Veronese} still holds. In such a case, $\kappa^{iso}_{\nn,\bomega}(Q)$ has $2^{|J|}$ irreducible components, where each one of them is isomorphic to the Kalman variety of the irreducible product $\prod_{j\notin J}Q_j$.
\end{Theorem}

\begin{proof}
A tensor $T\in\PP(S^\bomega V)$ is a hyperplane $H_T\subset \PP(S^\bomega V)^{*}$.
Recall that we identify all the vector spaces with their duals using the Frobenius inner product. 

Suppose that, for some $k$-tuple $([\xx_1],\ldots,[\xx_k])\in Q$,
the affine tangent space 
$\widehat{T}_{[\xx_1^{\omega_1}\otimes\cdots\otimes\xx_k^{\omega_k}]}v_{\nn,\bomega}(Q)$ at $[\xx_1^{\omega_1}\otimes\cdots\otimes\xx_k^{\omega_k}]$ to 
$v_{\nn,\bomega}(Q)$ is a subspace of the affine cone 
$\widehat{H}_T$ of $H_T$. 
Any vector $\vv$ of this affine tangent space may be written as
\[
\vv=\sum_{i=1}^k \xx_1^{\omega_1}\otimes\cdots\otimes\xx_i^{\omega_i-1}\cdot\yy_i\otimes\cdots\otimes\xx_k^{\omega_k}\,,\quad \yy_i\in \widehat{T}_{[\xx_i]}Q_i\quad\forall\,i\in[k]\,.
\]
By linearity, we may suppose that $\vv=\xx_1^{\omega_1}\otimes\cdots\otimes\xx_i^{\omega_i-1}\cdot\yy_i\otimes\cdots\otimes\xx_k^{\omega_k}$ for some $i\in[k]$. Then $\vv\in \widehat{H}_T$ for every choice of $\yy_i\in \widehat{T}_{[\xx_i]}Q_i$. 
This implies that the contraction $T(\xx_1^{\omega_1}\otimes\cdots\otimes\xx_i^{\omega_i-1}\cdot\yy_i\otimes\cdots\otimes\xx_k^{\omega_k})$ vanishes for every $\yy_i\in \widehat{T}_{[\xx_i]}Q_i$. Equivalently, the vector ${\bf w}_i\coloneqq T(\xx_1^{\omega_1}\otimes\cdots\otimes\xx_i^{\omega_i-1}\cdot\_\otimes\cdots\otimes\xx_k^{\omega_k})\in V_i$ is contained in the affine conormal space $\widehat{N}_{[\xx_i]}Q_i = \left(\widehat{T}_{[\xx_i]}Q_i\right)^\perp$. Observing that
\[
\widehat{N}_{[\xx_i]}Q_i = \left\{\yy_i\in V_i\mid\mathrm{rank}
\begin{pmatrix}
\nabla q_i(\xx_i)\\\yy_i
\end{pmatrix}
\le 1
\right\}
=
\left\{\yy_i\in V_i\mid\mathrm{rank}
\begin{pmatrix}
\xx_i\\\yy_i
\end{pmatrix}
\le 1
\right\}
=\langle\xx_i\rangle\,,
\]
we conclude that the vector ${\bf w}_i$ defined above is proportional to the vector $\xx_i$. Applying this argument for each $i\in [k]$, we find that $([\xx_1],\ldots,[\xx_k])$ is a singular vector $k$-tuple for the tensor $T$. This means that $T\in \kappa^{iso}_{\nn, \bomega}$.
Recall that
\[
[v_{\nn,\bomega}(Q)]^\vee=\overline{\{T\in\PP (S^\bomega V) \mid \mbox{$\widehat{T}_{[\xx]}v_{\nn,\bomega}(Q)\subset H_T$ for some $[\xx]\in v_{\nn,\bomega}(Q)$}\}}\ .
\]
Upon taking closures, we have proven the inclusion $[v_{\nn,\bomega}(Q)]^\vee\subset \kappa^{iso}_{\nn, \bomega}$. The variety $[v_{\nn,\bomega}(Q)]^\vee$ is a hypersurface for every choice of $\nn$ and $\bomega$ by Lemma \ref{lem: when Segre-Veronese embedding isotropic quadrics is hypersurface}. On the other hand, the totally isotropic Kalman variety  $\kappa^{iso}_{\nn, \bomega}$ is an irreducible hypersurface if $n_j\ge 3$ for all $j\in[k]$. Hence equality $\kappa^{iso}_{\nn, \bomega}=[v_{\nn,\bomega}(Q)]^\vee$ follows.

Since $[v_{\nn,\bomega}(Q)]^\vee$ is a hypersurface, its degree is equal to the polar class $\delta_0(v_{\nn,\bomega}(Q))$ by \cite[Theorem 3.4]{Holme}. This invariant can be computed using the Chern classes of $v_{\nn,\bomega}(Q)$ by the relation
\[
\delta_0(v_{\nn,\bomega}(Q))=\sum_{j=0}^{N}(-1)^j(N+1-j)c_j(v_{\nn,\bomega}(Q))\cdot h^{N-j}\,,
\]
where $h=c_1(\mathcal{O}_{v_{\nn,\bomega}(Q)}(1))$ is the hyperplane class.
This computation was done in \cite[Proposition 5.3.1]{sodphd} in a slightly more general setting and leads to the degree \eqref{eq: degree Kalman partially symmetric isotropic}.

Suppose $J\neq \emptyset$. For $1\leq j\leq 2^{|J|}$, define $Y_j$ to be the $j$th irreducible component of $Q$. Notice that each $Y_j$ is isomorphic to a product of quadrics.
We now describe $[v_{\nn,\bomega}(Q)]^\vee$ in terms of the dual varieties $[v_{\nn,\bomega}(Y_j)]^\vee$. 

By definition, $v_{\nn,\bomega}(Q) = \bigcup_{j\in J} v_{\nn,\bomega}(Y_j)$. Consider the conormal 
variety $\mathcal N_{v_{\nn,\bomega}(Q)}$ and the diagram 
\begin{equation}
\begin{tikzcd}
& \mathcal N_{v_{\nn,\bomega}(Q)}  \arrow{dl}[swap]{p_1} \arrow{dr}{p_2} & \\
v_{\nn,\bomega}(Q)& & \PP(S^\bomega V).
\end{tikzcd}
\end{equation}
The preimage under $p_1$ of each irreducible component $v_{\nn,\bomega}(Y_j)$ is a projective bundle and therefore irreducible \cite[Exercise 11.4.C]{Vakil}.
The image of such a component is the irreducible hypersurface $[v_{\nn,\bomega}(Y_j)]^{\vee}$.
This shows that $[v_{\nn,\bomega}(Q)]^\vee = \bigcup_{j\in J} [v_{\nn,\bomega}(Y_j)]^{\vee}$.
To see that these irreducible components are all distinct, assume $[v_{\nn,\bomega}(Y_j)]^{\vee} = [v_{\nn,\bomega}(Y_i)]^{\vee}$ for some $i\neq j\in J$.
Applying the dual construction again and the fact that every projective irreducible variety is reflexive in characteristic zero, we find that $v_{\nn,\bomega}(Y_j) = [v_{\nn,\bomega}(Y_j)]^{\vee\vee} = [v_{\nn,\bomega}(Y_i)]^{\vee\vee} = v_{\nn,\bomega}(Y_i)$, a contradiction, because $Y_i$ and $Y_j$ are distinct.  

Now, the totally isotropic Kalman variety is $\kappa^{iso}_{\nn, \bomega} = \bigcup_{j\in J} \kappa^{iso}_{\nn, \bomega}(Y_j)$, where we do not know a priori that all the irreducible varieties in the union are distinct. 
However, by the first part of this proof, we find that $\kappa^{iso}_{\nn, \bomega}(Y_j) = [v_{\nn,\bomega}(Y_j)]^{\vee}$.
Moreover, as we have checked, all the irreducible components $[v_{\nn,\bomega}(Y_j)]^{\vee}$ are distinct.
Therefore $\kappa^{iso}_{\nn, \bomega} = \bigcup_{j\in J}[v_{\nn,\bomega}(Y_j)]^{\vee} = [v_{\nn,\bomega}(Q)]^\vee$ has $2^{|J|}$ components, where each of them is a dual variety of the corresponding irreducible quadric. 
\end{proof}

\begin{Example}
Let $k=2$ and $\nn = (2,2)$. We assume that $\PP(V_1)=\PP(V_2)\cong\PP_{\mC}^1$ and $Q_1=Q_2=\mathcal{V}(x_1^2+x_2^2)\subset\PP_{\mC}^1$. One has 
\[
\kappa_{\nn}(\PP_{\mC}^1\times Q_2) = \kappa^{iso}_{\nn} = \kappa_{\nn}(Q_1\times\PP_{\mC}^1)\,.
\]
This is a reducible hypersurface of degree $4$ whose irreducible components are four planes in $\PP(\C^2\otimes\C^2)\cong\PP_{\mC}^3$, dual to the four points of $Q_1\times Q_2$ as predicted by Theorem \ref{thm: codim degree Kalman partially symmetric isotropic}. Denoting $a_{ij}$ the homogeneous coordinates of $\PP_{\mC}^3$, it is defined by the vanishing of the polynomial $\left[(a_{11}-a_{22})^2 + (a_{12}+a_{21})^2\right]\left[(a_{11}+a_{22})^2 + (a_{12}-a_{21})^2\right]$.
There are four lines in $\PP_{\mC}^3$ obtained as the intersection of the four pairs of non-conjugate planes. Each of the Kalman strata $\kappa^{\{1\}}_{\nn}(Q_1\times \PP_{\mC}^1)=\kappa^{\{1\}}_{\nn}(\PP_{\mC}^1\times \PP_{\mC}^1)$ and $\kappa^{\{2\}}_{\nn}(\PP_{\mC}^1\times Q_2)=\kappa^{\{2\}}_{\nn}(\PP_{\mC}^1\times \PP_{\mC}^1)$ consists of two such lines. Therefore, both these Kalman strata have degree $2$ and codimension $2$. 
\end{Example}

In the symmetric case, Theorem \ref{thm: codim degree Kalman partially symmetric isotropic} specializes to the following result, which is related to \cite[Proposition 2.10]{BGV}.

\begin{Theorem}\label{thm: codim degree Kalman symmetric isotropic}
Let $\kappa^{iso}_{n, \omega}$ be the Kalman variety of symmetric tensors having an isotropic eigenvector. Let $v_{n,\omega}\colon\PP(V)\to\PP(S^\omega V)$ be the degree-$\omega$ Veronese embedding of $\PP(V)$. Then
\[
\kappa^{iso}_{n, \omega}=[v_{n,\omega}(Q)]^\vee\,.
\]
In particular, $\kappa^{iso}_{n, \omega}$ is an irreducible hypersurface of $\PP(S^\omega V)$. Its degree is
\begin{equation}\label{eq: degree Kalman isotropic}
\deg(\kappa^{iso}_{n, \omega})=2\sum_{j=0}^{n-2}(j+1)(\omega-1)^j\,.
\end{equation}
\end{Theorem}

\begin{Example}
Consider the totally isotropic Kalman variety $\kappa^{iso}_{3,2}$ of ternary quadrics possessing an isotropic eigenvector. By Theorem \ref{thm: codim degree Kalman symmetric isotropic}, this is an irreducible hypersurface of degree $6$ in $\PP(S^\omega V)\cong\PP^5$. We include a \texttt{Macaulay2} code to verify this, using both descriptions according to Theorem \ref{thm: codim degree Kalman symmetric isotropic}. Below we denote the homogeneous coordinates of $\PP(S^\omega V)$ by $a_0,\ldots,a_5$.

As Kalman variety:
{\small
\begin{Verbatim}[commandchars=\\\{\}]
R = \textcolor{purple}{QQ}[x_1..x_3,a_0..a_5,c_0..c_5];
MX = \textcolor{ForestGreen}{matrix}\{\{x_1^2,2*x_1*x_2,2*x_1*x_3,x_2^2,2*x_2*x_3,x_3^2\}\};
aa = \textcolor{ForestGreen}{matrix}\{\{a_0..a_5\}\}; xx = \textcolor{ForestGreen}{matrix}\{\{x_1..x_3\}\};
f = (MX*\textcolor{ForestGreen}{transpose}(aa))_(0,0);
Ivect = \textcolor{ForestGreen}{minors}(2,\textcolor{ForestGreen}{diff}(xx,f)||xx);
IQ = \textcolor{ForestGreen}{ideal}(\textcolor{ForestGreen}{sum}(3, i-> x_(i+1)^2));
sat = \textcolor{ForestGreen}{saturate}(Ivect+IQ,\textcolor{ForestGreen}{ideal} xx);
Kalman = \textcolor{ForestGreen}{eliminate}(\textcolor{ForestGreen}{first entries} xx,sat);
\end{Verbatim}
}
As dual variety $[v_{3,2}(Q)]^\vee$:
{\small
\begin{Verbatim}[commandchars=\\\{\}]
cc = \textcolor{ForestGreen}{matrix}\{\{c_0..c_5\}\};
IVQ = \textcolor{ForestGreen}{eliminate}(\textcolor{ForestGreen}{first entries} xx,
      \textcolor{ForestGreen}{saturate}(\textcolor{ForestGreen}{ideal}(\textcolor{ForestGreen}{first entries}(MX-aa))+IQ,\textcolor{ForestGreen}{ideal} xx));
jacIVQ = \textcolor{ForestGreen}{diff}(aa, \textcolor{ForestGreen}{transpose gens} IVQ);
norIVQ = \textcolor{ForestGreen}{saturate}(IVQ + \textcolor{ForestGreen}{minors}(\textcolor{ForestGreen}{codim}(IVQ)+1,jacIVQ||cc), \textcolor{ForestGreen}{ideal} aa); 
dualIVQ = \textcolor{ForestGreen}{eliminate}(\textcolor{ForestGreen}{first entries} aa, norIVQ);
eq = \textcolor{ForestGreen}{sub}(dualIVQ_0, \textcolor{ForestGreen}{apply}(6, i-> c_i=>a_i));
eq == Kalman_0 \textcolor{orange}{-- the two equations coincide}
\end{Verbatim}
}
We point that the defining polynomial \verb+eq+ of $[v_{n,\omega}(Q)]^\vee$  has a role in the theory of characteristic polynomials of symmetric tensors. The second author showed that the leading coefficient of the characteristic polynomial of a symmetric tensor in $S^\omega V$ is equal, up to scaling, to the defining polynomial of $[v_{n,\omega}(Q)]^\vee$ with multiplicity $\omega-2$ \cite{Sod}. 
\end{Example}

\section{Generating functions}\label{sec: degreesofKalmans}

\epigraph{{\it A generating function is a device somewhat similar to a bag.\\Instead of carrying many little objects detachedly, which could be embarrassing, we put them all in a bag, and then we have only one object to carry, the bag}.}{George P\'{o}lya \cite[Chapter VI]{GP}}

In the previous section, we determined the degrees of the generalized Kalman varieties $\kappa_{\nn, \bomega}(Z)$ by computing the coefficients $d(\nn,\bdelta,\bomega)$. In the spirit of P\'{o}lya's quote, Theorem \ref{thm: generating function, partially symmetric case} furnishes a generating function for the coefficients $d(\nn,\bdelta,\bomega)$ when $\bdelta=(\delta,0,\ldots,0)$, or equivalently when we consider only one subvariety $Z_1\subset\PP(\C^{n_1})$ of codimension $\delta$. For this particular choice of $\bdelta$, we use the notation $d(\nn,\delta,\bomega)$ in place of $d(\nn,\bdelta,\bomega)$.
Before proceeding to the proof of Theorem \ref{thm: generating function, partially symmetric case}, we start with a lemma.

\begin{Lemma}\label{lem: summation}
Let $d(\mm)=\sum_{j_1=0}^{m_1-1}\cdots\sum_{j_k=0}^{m_k-1}f(j_1,\ldots, j_k)$. Then
\[
\sum_{\mm\in\mathbb{N}^k} d(\mm)\,\xx^\mm\,=\left(\prod_{i=1}^k\frac{x_i}{1-x_i}\right)\sum_{{\bf j}\in\mathbb{N}^k}f({\bf j})\,\xx^{\bf j}\,.
\]
\end{Lemma}
\begin{proof} 
The proof is an induction on $k$. For $k=1$, we have 
\begin{align*}
\sum_{m=0}^\infty \left(\sum_{j=0}^{m-1}f(j)\right)x^{m} &= f(0)x+\left(f(0)+f(1)\right)x^2+\left(f(0)+f(1)+f(2)\right)x^3+\cdots\\
&= (x+x^2+x^3+\cdots)\left(f(0)+f(1)x+f(2)x^2+\cdots\right)\\ &= \frac{x}{1-x}\sum_{j=0}^\infty f(j)x^j\,.
\end{align*}
The induction step is similar. 
\end{proof}

As in Zeilberger's approach \cite{EZ}, we shall employ a classical and powerful theorem of MacMahon \cite[\S 3, Chapter 2, 66]{MacMahon}. 

\begin{Theorem}[MacMahon Master Theorem]\label{thm: MMT}
Let $A=(a_{ij})$ be an $m\times m$ complex matrix, and let $\zz=(z_1,\ldots,z_m)$ be a vector of formal variables. Let $f({\bf p})$ be the coefficient of the monomial $\zz^{\bf p}$ in the product $\prod _{i=1}^{m}(a_{i1}z_{1}+\cdots+a_{im}z_{m})^{p_{i}}$.
Let $\mathbf{w}=(w_{1},\ldots,w_{m})$ be another vector of formal variables, $T=\mathrm{diag}(\bf{w})$ and denote by $I_m$ the identity matrix of size $m$. Then
\[
\sum_{{\bf p}\in\mathbb{N}^m} f({\bf p})\,{\bf w}^{\bf p}\,=\,{\frac {1}{\det(I_{m}-TA)}}\,.
\]
\end{Theorem}

\noindent We are now ready to prove Theorem \ref{thm: generating function, partially symmetric case}. 

\begin{proof}[Proof of Theorem \ref{thm: generating function, partially symmetric case}]
Recall that, for any vector of codimensions $\bdelta$, the integer $d(\nn,\bdelta,\bomega)$ is the coefficient of the monomial $h^\delta\prod_{i=1}^kt_{i}^{n_{i}-\delta_i-1}$ in the polynomial
\[
\prod_{i=1}^{k}\frac{(\widehat{t_i}+h)^{n_{i}}-t_i^{n_i}}{(\widehat{t_i}+h)-t_i}=\prod_{i=1}^{k}\sum_{j_i=0}^{n_i-1}(\widehat{t_i}+h)^{j_{i}}t_i^{n_i-1-j_i}\,,
\]
where $\delta=\sum_{i=1}^k\delta_i$ and $\widehat{t_i}\coloneqq\left(\sum_{j=1}^{k}\omega_j t_j\right)- t_i$. Equivalently, $d(\nn,\bdelta,\bomega)$ is the constant term of
\[
\sum_{j_1=0}^{n_1-1}\cdots \sum_{j_k=0}^{n_k-1}h^{-\delta}\prod_{i=1}^{k}(\widehat{t_i}+h)^{j_{i}}t_i^{\delta_i-j_i}\,,
\]
namely the sum of the constant terms in the products $h^{-\delta}\prod_{i=1}^{k}(\widehat{t_i}+h)^{j_{i}}t_i^{\delta_i-j_i}$. Observe that the constant term of each product $h^{-\delta}\prod_{i=1}^{k}(\widehat{t_i}+h)^{j_{i}}t_i^{\delta_i-j_i}$ is the coefficient of $h^\delta\prod_{i=1}^k t_i^{j_i}$ in $\prod_{i=1}^k(\widehat{t_i}+h)^{j_i}t_i^{\delta_i}$, which we may call $f({\bf j},\bdelta)$. Therefore
\[
d(\bdelta,\nn, \bomega)=\sum_{j_1=0}^{n_1-1}\cdots \sum_{j_k=0}^{n_k-1}f({\bf j},\bdelta)\,.
\]
\noindent Applying Lemma \ref{lem: summation} with respect to $\nn$, we have
\[
\sum_{\nn,\bdelta\in\mathbb{N}^k} d(\bdelta,\nn, \bomega)\,\xx^\nn \yy^\bdelta = \prod_{i=1}^k\frac{x_i}{1-x_i}\sum_{{\bf j},\bdelta\in\mathbb{N}^k} f({\bf j},\bdelta)\,\xx^{\bf j}\yy^{\bdelta}. 
\]
From now on, we assume that $\delta_1=\delta$ and $\delta_i=0$ for all $i\ge 2$ and we call $f({\bf j},\delta)$ the coefficient of $h^\delta\prod_{i=1}^k t_i^{j_i}$ in $t_1^\delta\prod_{i=1}^k(\widehat{t_i}+h)^{j_i}$. In this special case, we can apply MacMahon's Theorem \ref{thm: MMT} with $m=k+1$, $f({\bf j},\delta)$ in place of $f({\bf p})$, and considering the product
\[
t_1^{\delta}\prod_{i=1}^k(\widehat{t_i}+h)^{j_i} = t_1^\delta\prod _{i=1}^{k}(a_{i,1}t_{1}+\cdots+a_{i,k}t_{k}+a_{i,k+1}h)^{j_{i}}\,.
\]
Here $A=(a_{ij})$ is the $(k+1)\times (k+1)$ matrix defined as
\[
A\coloneqq
\left(\begin{array}{cccc|c}
&&&&1\\
&&\hspace{-20pt}B&&\vdots\\
&&&&1\\ \hline
1&0&\cdots&0&0
\end{array}\right), 
\]
where $B$ is the $k\times k$ matrix whose $(i,j)$-th entry is $\omega_j-\delta_{ij}$ ($\delta_{ij}=1$ if $i=j$, and $0$ otherwise). Then
\[
\sum_{{\bf j}\in\mathbb{N}^k}\sum_{\delta=0}^\infty f({\bf j},\delta)\,\xx^{\bf j}y^{\delta} = \frac{1}{\det(I_{k+1}-TA)}\,,
\]
where $T=\mathrm{diag}(\xx,y)$. Summing up, we have obtained the formula
\begin{equation}\label{eq: identity generating function MacMahon}
\sum_{\nn\in\mathbb{N}^k}\sum_{\delta=0}^{\infty} d(\nn,\delta,\bomega)\,\xx^\nn y^\delta\,=\,\prod_{i=1}^k\frac{x_i}{1-x_i}\frac{1}{\det(I_{k+1}-TA)}\,.
\end{equation}

\noindent It remains to compute $\det(I_{k+1}-TA)$. Define $M\coloneqq I_{k+1}-TA$. More explicitly 
\[
M =
\left(
\begin{array}{c|cccccc}
1-(\omega_1-1)x_1 & -\omega_2x_1 & -\omega_3x_1 & \cdots & -\omega_{k-1}x_1 & -\omega_k x_1 & -x_1 \\
-\omega_1x_2 & 1-(\omega_2-1)x_2 & -\omega_3x_2 & \cdots  & -\omega_{k-1}x_2& -\omega_k x_2& -x_2 \\
\vdots & \vdots & \vdots & & \vdots & \vdots & \vdots \\
-\omega_1x_k & -\omega_2x_k & -\omega_3x_k & \cdots  & -\omega_{k-1}x_k & 1-(\omega_k-1)x_k & -x_k \\
\hline\\[-10pt]
-y & 0 & 0 & \cdots & 0 & 0 & 1 \\
\end{array}
\right)
\,.
\]
By expanding along the last row, we derive two $k \times k$ submatrices
\begin{align*}
M' & \coloneqq
\begin{pmatrix}
-\omega_2x_1 & -\omega_3x_1 & \cdots & -\omega_kx_1 & -x_1 \\
1-(\omega_2-1)x_2 & -\omega_3x_2 &\cdots & -\omega_kx_2 & -x_2 \\
\vdots & \vdots & & \vdots & \vdots \\
-\omega_2x_k & -\omega_3x_k & \cdots & 1-(\omega_k-1)x_k & -x_k \\
\end{pmatrix}\,,\\
M'' & \coloneqq
\begin{pmatrix}
1-(\omega_1-1)x_1 & -\omega_2x_1 & \cdots & -\omega_kx_1 \\
-\omega_1x_2 & 1-(\omega_2-1)x_2 &\cdots & -\omega_kx_2 \\
\vdots & \vdots & & \vdots \\
-\omega_1x_k & -\omega_2x_k & \cdots & 1-(\omega_k-1)x_k \\
\end{pmatrix}\,.
\end{align*}
The next step is to compute the determinants of $M'$ and $M''$.

\smallskip
\noindent {\bf Claim 1:} We have
\begin{equation}\label{eq: formula det M'}
\det(M')=(-1)^k x_1\prod_{i=2}^k(1+x_i)\,.
\end{equation}
First of all, the polynomial $\det(M')$ has degree $\le k$ in $x_1,\ldots,x_k$. The first row of $M'$ is a multiple of $x_1$, hence $x_1$ divides $\det(M')$. Note that, if $x_i=-1$ then the first and the $i$-th row of $M'$ are proportional and so $\deg(M')$ vanishes. By Euclidean division, we may write $\det(M')=(1+x_i)q + R$, where $R$ is the remainder, which then depends only on the variables $x_j$ with $j \neq i$. Hence $R(x_1, \ldots, \widehat{x_i}, \ldots x_k)=0$ for every value of the variables. Since we are in characteristic zero, $R=0$. The left-hand side is divisible by $x_1\prod_{i \geq 2}(1+ x_i)$ of degree $k$, and hence there is a scalar $\gamma$ such that $\det(M')= \gamma\,x_1\prod_{i \geq 2}(1+ x_i)$. In order to determine that $\gamma=(-1)^k$, we substitute $x_1=1$ and $x_2=\cdots=x_k=0$ in $M'$ and observe that the determinant of the resulting matrix is $(-1)^k$.

\smallskip
\noindent {\bf Claim 2:} We have
\begin{equation}\label{eq: formula det M''}
\det(M'')= \prod_{i=1}^k(1+x_i)-\sum_{j=1}^k \omega_jx_j\prod_{i \neq j}(1+x_i)\,.
\end{equation}
Write $M'' = D - \xx^\mathsmaller{T}\bomega$, where $\xx=(x_1,\ldots,x_k)$ and $\bomega=(\omega_1,\ldots,\omega_k)$ are row vectors and $D=\mathrm{diag}(\bf{1}+\xx)$ is the diagonal matrix whose $i$-th diagonal entry is $1+x_i$. 
Since $D$ is invertible over the fraction field $\C(\xx)$, we have $M'' = D[I_k - (D^{-1}\xx^\mathsmaller{T})\,\bomega]$.
By the matrix determinant lemma \cite[\S 4.1, Fact 22]{Hogben} over the fraction field $\C(\xx)$, one has
\[
\det(I_k - (D^{-1}\xx^\mathsmaller{T})\,\bomega) = 1 - \bomega D^{-1}\xx^\mathsmaller{T}\,. 
\]
Thus
\[
\det(M'') = \det(D - \xx^\mathsmaller{T}\bomega) = (1 - \bomega D^{-1} \xx^\mathsmaller{T})\det(D) = \prod_{i=1}^k(1+x_i)-\sum_{j=1}^k \omega_jx_j\prod_{i \neq j}(1+x_i)\,.
\]
From the formulas \eqref{eq: formula det M'} and \eqref{eq: formula det M''}, we conclude that
\begin{align*}
\det(I_{k+1}-TA) & = \det(M) = (-1)^{k+1}y\det(M')+\det(M'') = \\
& = -y\,x_1\prod_{i=2}^k(1+x_i)+\prod_{i=1}^k(1+ x_i)- \sum_{j=1}^k \omega_jx_j\prod_{i \neq j}(1+ x_i)\,,
\end{align*}
which establishes the result.\qedhere
\end{proof}

\begin{Example}
For $k=2$ and $\bomega=(1,1)$ we have 
\begin{align*}
\sum_{\nn\in\mathbb{N}^k}\sum_{\delta=0}^{\infty} d(\nn,\delta,\bomega)\,\xx^\nn y^\delta &= \frac{x_1}{1-x_1}\cdot\frac{x_2}{1-x_2}\cdot\left|\begin{matrix}1&-x_1&-x_1\\-x_2&1&-x_2\\-y&0&1\end{matrix}\right|^{-1}\\
&= \frac{x_1}{1-x_1}\cdot\frac{x_2}{1-x_2}\cdot\frac{1}{1-x_1y-x_1x_2-x_1x_2y}=2y x_1^2x_2^2+3y^2x_1^3x_2^2+\cdots
\end{align*}
This fits with \cite[Table 1]{OSh}.
Note that for $y=0$, when substituted into the first expression of the second line, we recover the formula in \cite[Theorem 1]{FO}. 
\end{Example}

\begin{Remark}\label{rmk: simplify det M' and det M''}
If $\omega_1=\cdots=\omega_k=\omega$, the polynomial $H_\bomega(\xx,y)$ may be rewritten more symmetrically as
\begin{align*}
H_{\bomega}(\xx,y) &= -y\,x_1\prod_{i=2}^k(1+x_i)+\prod_{i=1}^k(1+ x_i)-\omega\sum_{j=1}^k x_j\prod_{i \neq j}(1+ x_i)\\
&= -y\,x_1\sum_{i=0}^{k-1}e_i(\widehat{\xx}_1)+\sum_{i=0}^k e_i(\xx)-\omega\sum_{j=0}^{k-1}\sum_{i=1}^k x_i e_j(\widehat{\xx}_i)\\
&= -y\,x_1\sum_{i=0}^{k-1}e_i(\widehat{\xx}_1)+\sum_{i=0}^{k}(1-\omega i)e_i(\xx)\,,
\end{align*}
where $\widehat{\xx}_i$ denotes the vector $(x_1,\ldots,x_{i-1},x_{i+1},\ldots,x_k)$ and for a vector $\xx=(z_1,\ldots,z_k)$, we define $e_0(\xx)\coloneqq 1$ and $e_i(\xx)\coloneqq\sum_{1\le j_1<\cdots<j_i\le k}z_{j_1}\cdots z_{j_i}\quad\forall\,i\in[k]$. 
\end{Remark}

\section{Asymptotics}\label{Asymptotics}

\subsection{Asymptotic behavior of \texorpdfstring{$d(\nn,\bdelta,\bomega)$}{c delta nk} for \texorpdfstring{$n_i\to\infty$}{ni}}
Fix an index $i\in[k]$. We study the asymptotic behaviour of $d(\nn,\bdelta,\bomega)$ when the dimension $n_i$ goes to infinity. If $\delta_i=0$ and $\bomega=(1,\ldots,1)$, we recover \cite[Corollary 3.5]{OSh}. If $\delta_i\ge 1$ and $\omega_i=1$, regardless of the other entries of $\bdelta$ and $\bomega$, we obtain a new stabilization property.

\begin{Proposition}\label{prop: stabilization degree Kalman}
Fix $i\in[k]$ and assume that $\omega_i=1$. Let $n_i-1=\sum_{j\neq i}(n_j-1)+\delta_i$. Then
\[
d(n_1,\ldots,n_{i-1},m,n_{i+1},\ldots,n_k,\bdelta,\bomega)=d(n_1,\ldots,n_{i-1},n_i,n_{i+1},\ldots,n_k,\bdelta,\bomega)\quad\forall m\ge n_i\,.
\]
\end{Proposition}

\begin{proof}
Recall that $\delta=\sum_{i=1}^k\delta_i$. Let $T_i=\prod_{j\neq i} t_j^{n_j-\delta_j-1}$. We need to compare the coefficients of
\begin{align*}
&(1)\ h^\delta t_i^{m-\delta_i-1}T_i\quad\text{in}\quad\frac{(\widehat{t_i}+h)^m-t_i^m}{\widehat{t_i}+h-t_i}\prod_{j\neq i}\frac{(\widehat{t_j}+h)^{n_j}-t_j^{n_j}}{\widehat{t_j}+h-t_j}=\left[\sum_{l=0}^{m-1}(\widehat{t_i}+h)^{m-1-l}t_i^l\right]\prod_{j\neq i}\frac{(\widehat{t_j}+h)^{n_j}-t_j^{n_j}}{\widehat{t_j}+h-t_j}\\
&(2)\ h^\delta t_i^{n_i-\delta_i-1}T_i\quad\text{in}\quad\frac{(\widehat{t_i}+h)^{n_i}-t_i^{n_i}}{\widehat{t_i}+h-t_i}\prod_{j\neq i}\frac{(\widehat{t_j}+h)^{n_j}-t_j^{n_j}}{\widehat{t_j}+h-t_j}=\left[\sum_{l=0}^{n_i-1}(\widehat{t_i}+h)^{n_1-1-l}t_i^l\right]\prod_{j\neq i}\frac{(\widehat{t_j}+h)^{n_j}-t_j^{n_j}}{\widehat{t_j}+h-t_j}\,.
\end{align*}
Observe that we can replace $(2)$ with the coefficient of $t_i^{m-n_i}h^\delta t_i^{n_i-\delta_i-1}T_i=h^\delta t_i^{m-\delta_i-1}T_i$ in
\[
\left[\sum_{l=0}^{n_i-1}(\widehat{t_i}+h)^{n_i-1-l}t_i^{l+m-n_i}\right]\prod_{j\neq i}\frac{(\widehat{t_j}+h)^{n_j}-t_j^{n_j}}{\widehat{t_j}+h-t_j}= \left[\sum_{s=m-n_i}^{m-1}(\widehat{t_i}+h)^{m-1-s}t_i^s\right]\prod_{j\neq i}\frac{(\widehat{t_j}+h)^{n_j}-t_j^{n_j}}{\widehat{t_j}+h-t_j}\,.
\]

\noindent Now consider the product
\begin{equation}\label{eq: product s from 0 to m-n_1-1}
\left[\sum_{s=0}^{m-n_i-1}(\widehat{t_i}+h)^{m-1-s}t_i^s\right]\prod_{j\neq i}\frac{(\widehat{t_j}+h)^{n_j}-t_j^{n_j}}{\widehat{t_j}+h-t_j}\,.
\end{equation}
Recall that $\widehat{t_i}\coloneqq\left(\sum_{j=1}^{k}\omega_j t_j\right)- t_i$. Since $\omega_i=1$, the maximum degree of $t_i$ in the first factor of \eqref{eq: product s from 0 to m-n_1-1} is $m-n_i-1$. Moreover, the maximum degree of $t_i$ in the second factor of \eqref{eq: product s from 0 to m-n_1-1} is $\sum_{j\neq i}(n_j-1)$. Summing up, the maximum total degree of $t_i$ in \eqref{eq: product s from 0 to m-n_1-1} is
\[
m-n_i-1+\sum_{j\neq i}(n_j-1)=m-\delta_i-1+\left(\delta_i+\sum_{j\neq i}(n_j-1)-n_i\right)<m-\delta_i-1\,
\]
because $\delta_i+\sum_{j\neq i}(n_j-1)-n_i<0$ by hypothesis.
Therefore, the product \eqref{eq: product s from 0 to m-n_1-1} gives no contribution to the coefficient of $h^\delta t_i^{m-\delta_i-1}T_i$. This implies the stabilization.
\end{proof}

\begin{Remark}
Proposition \ref{prop: stabilization degree Kalman} does not hold if we study the asymptotic behaviour of $d(\nn,\bdelta,\bomega)$ when $n_i\to\infty$ for $\omega_i>1$. An immediate counterexample can be found in the symmetric case $k=1$, see Remark \ref{rmk: degree Kalman symmetric case}.
\end{Remark}

\subsection{Asymptotic behavior of \texorpdfstring{$d(\nn,\bdelta,\bomega)$}{c delta nk} in the binary format for \texorpdfstring{$k\to\infty$}{k to infinity}}

Assume $n_1=\cdots=n_k=2$. Here $\delta_i\in\{0,1\}$ for all $i\in[k]$. By \cite[Theorem 1.2]{OSh}, the integer $d(\nn,\bdelta,\bomega)$ is the coefficient of $h^\delta\prod_{i=1}^kt_i^{1-\delta_i}$ in
\[
(\omega_1t_1+\cdots+\omega_kt_k+h)^k = \sum_{j_1+\cdots+j_k+l=k}\binom{k}{j_1,\ldots,j_k,l}(\omega_1t_1)^{j_1}\cdots (\omega_kt_k)^{j_k}h^l\,,
\]
where $\binom{k}{j_1,\ldots,j_k,l}$ is the multinomial coefficient.
Setting $\mathcal{P}_0\coloneqq\{i\in[k]\mid\delta_i=0\}$, we have that $|\mathcal{P}_0|=k-\delta$ and $d(\nn,\bdelta,\bomega)=\binom{k}{\delta}\prod_{i\in\mathcal{P}_0}\omega_i$. In particular, the growth of $d(\nn,\bdelta,\bomega)$ is factorial in $k$.

\subsection{Asymptotic behavior of \texorpdfstring{$d(\nn,\delta,\bomega)$}{c delta nk} in the hypercubical format \texorpdfstring{$n^k$}{n k} for \texorpdfstring{$n\to\infty$}{n}}

In Theorem \ref{thm: generating function, partially symmetric case} we derived a generating function for the integers $d(\nn,\delta,\bomega)$, hence for $\bdelta=(\delta,0,\ldots,0)$. Observe that the polynomial $H_\bomega(\xx,y)$ in \eqref{eq: def H, partially symmetric} is of the form $H_\bomega(\xx,y)=-H_{1,\bomega}(\xx)y+H_{2,\bomega}(\xx)$.
Therefore, we can rewrite the generating function of the degrees $d(\nn,\delta,\bomega)$ as
\begin{equation}\label{eq: generating function delta fixed}
\sum_{\nn\in\mathbb{N}^k}\sum_{\delta=0}^{\infty} d(\nn,\delta,\bomega)\,\xx^\nn y^\delta = \frac{1}{-H_{1,\bomega}(\xx)y+H_{2,\bomega}(\xx)}\prod_{i=1}^k\frac{x_i}{1-x_i}=\sum_{\delta=0}^{\infty}F(\xx)\,y^\delta\,,
\end{equation}
where
\begin{equation}\label{eq: def F}
F_\bomega(\xx)\coloneqq \frac{H_{1,\bomega}(\xx)^\delta}{H_{2,\bomega}(\xx)^{\delta+1}}\prod_{i=1}^k\frac{x_i}{1-x_i}=\frac{H_{1,\bomega}(\xx)^\delta\prod_{i=1}^k x_i(1-x_i)^\delta}{\left[H_{2,\bomega}(\xx)\prod_{i=1}^k(1-x_i)\right]^{\delta+1}}=\frac{F_{N,\bomega}(\xx)}{[F_{D,\bomega}(\xx)]^{\delta+1}}\,.
\end{equation}
From now on, we restrict to the case $\bomega=\omega{\bf 1}$. By Remark \ref{rmk: simplify det M' and det M''} we have
\[
H_{1,\omega{\bf 1}}(\xx) = x_1\sum_{i=0}^{k-1}e_i(\widehat{\xx}_1)\,,\quad H_{2,\omega{\bf 1}}(\xx)=\sum_{i=0}^{k}(1-\omega i)e_i(\xx)\,.
\]
In this case, the reduced denominator $F_{D,\omega{\bf 1}}(\xx)$ of $F_{\omega{\bf 1}}(\xx)$ is symmetric with respect to the variables $x_i$ and coincides with the denominator of the generating function obtained in \cite[Proposition 1]{EZ} when $\omega=1$. The differences with \cite[Proposition 1]{EZ} are the numerator and the exponent $\delta+1$ in the denominator.

Our goal is to fix $\delta$ and study the asymptotic behaviour of $d(n,\delta,\omega)\coloneqq d(n{\bf 1},\delta,\omega{\bf 1})$ for $n\to\infty$. This can be done applying the next result by Raichev and Wilson \cite{RW}. We refer to that paper for the definitions of strictly minimal, critical, isolated, and non-degenerate point needed in the statement.

\begin{Theorem}{\cite[Theorem 3.2]{RW}}\label{thm: Raichev-Wilson}
Let $k\geq 2$ and let $G = G_N/G_D^{\delta+1}$, whose Taylor expansion in a neighborhood of the origin is $\sum_{\balpha\in \mathbb{N}^k} g_{\balpha}\,\xx^\balpha$. Suppose $\cc\in \mathcal{V} = \{ G_D(\xx) = 0 \}$ is smooth with $c_k\partial_k G_D(\cc) \neq 0$, strictly minimal, critical, isolated, and non-degenerate. Then, for all $N\in\mathbb{N}$, as $n\to+\infty$,
\[
g_{n {\bf 1}} = \cc^{-n{\bf 1}} \left[ \left((2\pi n)^{k-1} \det \tilde{g}''(0) \right)^{-1/2}\sum_{i=0}^{\delta}\sum_{j<N} \frac{(n+1)^{\overline{\delta-i}}}{(\delta-i)!\,i!}n^{-j} L_j(\tilde{u}_i, \tilde{g}) + O\left(n^{\delta-(k-1)/2-N}\right)\right]\,,
\]
where $r^{\overline{s}}\coloneqq r(r+1)\cdots(r+s-1)$. In the original formula of Raichev-Wilson, we substituted $\balpha = {\bf 1}$ and $p=\delta+1$. 
\end{Theorem}

The functions $\tilde{g}$, $\tilde{u}_i$ and $L_j$ appearing in Theorem \ref{thm: Raichev-Wilson} are defined in \cite[Definition 3.1]{RW} and in the statement of \cite[Theorem 3.2]{RW}.

\begin{Proposition}\label{prop: critical point Pantone}
The point $\cc=\left(\frac{1}{\omega k-1},\ldots,\frac{1}{\omega k-1}\right)$ is smooth for $\mathcal{V} = \{ F_{D,\omega{\bf 1}}(\xx) = 0 \}$ with $c_k\partial_k F_{D,\omega{\bf 1}}(\cc)\neq 0$, strictly minimal, critical, isolated, and non-degenerate.
\end{Proposition}
\begin{proof}
The statement is an immediate consequence of \cite[Remark 3.10]{OSV} and \cite[Propositions 2.2,\ldots,2.7]{Pan}.
\end{proof}

\begin{Proposition}\label{prop: L_0}
The following identity holds true:
\[
L_0(\tilde{u}_0, \tilde{g}) = \tilde{u}_0(\cc) = \frac{(\omega k-1)^{k-\delta-1}}{(\omega k)^{k-\delta-2}(\omega k-2)^k}\,.
\]
\end{Proposition}
\begin{proof}
Recall that $L_j(\tilde{u}_i,\tilde{g})$ is defined in \cite[Theorem 3.2]{RW}. For $j=0$, we have $L_0(\tilde{u}_i,\tilde{g}) = \tilde{u}_i(\cc)$. 
By \cite[Proposition 4.3]{RW}, we have $\tilde{u}_0(\cc) = \frac{F_{N,\omega{\bf 1}}(\cc)}{(-c_k \partial_k F_{D,\omega{\bf 1}}(\cc))^{\delta+1}}$, where
\begin{align*}
F_{N,\omega{\bf 1}}(\cc) &= \left[c_1\sum_{i=0}^{k-1}e_i(\widehat{\cc}_1)\right]^\delta\prod_{i=1}^k c_i(1-c_i)^\delta\\
&= \left[\frac{1}{\omega k-1}\sum_{i=0}^{k-1}\binom{k-1}{i}\left(\frac{1}{\omega k-1}\right)^i\right]^\delta\left(\frac{1}{\omega k-1}\right)^k\left(\frac{\omega k-2}{\omega k-1}\right)^{\delta k}\\
&= \left[\frac{1}{\omega k-1}\left(\frac{\omega k}{\omega k-1}\right)^{k-1}\right]^\delta\left(\frac{1}{\omega k-1}\right)^k\left(\frac{\omega k-2}{\omega k-1}\right)^{\delta k} = \frac{(\omega k)^{(k-1)\delta}(\omega k-2)^{k\delta}}{(\omega k-1)^{k(2\delta+1)}}\,,
\end{align*}
and, similarly as in the proof of \cite[Proposition 2.7]{Pan}, $-c_k\partial_k F_{D,\omega{\bf 1}}(\cc) = \frac{(\omega k)^{k-2}(\omega k-2)^k}{(\omega k-1)^{2k-1}}$. 
Hence
\[
\tilde{u}_0(\cc) = \frac{F_{N,\omega{\bf 1}}(\cc)}{(-c_k \partial_k F_{D,\omega{\bf 1}}(\cc))^{\delta+1}} = \frac{(\omega k)^{(k-1)\delta}(\omega k-2)^{k\delta}}{(\omega k-1)^{k(2\delta+1)}}\left[\frac{(\omega k-1)^{2k-1}}{(\omega k)^{k-2}(\omega k-2)^k}\right]^{\delta+1} = \frac{(\omega k-1)^{k-\delta-1}}{(\omega k)^{k-\delta-2}(\omega k-2)^k}\,.\qedhere
\]
\end{proof}

\noindent We have the necessary tools to prove Theorem \ref{thm: deg Kalman approx}.

\begin{proof}[Proof of Theorem \ref{thm: deg Kalman approx}]
By Theorem \ref{thm: Raichev-Wilson}, for $N = 1$ and $\cc = (\frac{1}{\omega k-1},\ldots, \frac{1}{\omega k-1})$, we have
\[
d(n,\delta,\omega) = (\omega k-1)^{kn} \left[\frac{1}{(2\pi n)^{\frac{k-1}{2}}\det\tilde{g}''(0)^{\frac{1}{2}}}\sum_{i=0}^{\delta}\frac{(n+1)^{\overline{\delta-i}}}{(\delta-i)!\,i!} L_0(\tilde{u}_i, \tilde{g}) + O\left(\frac{1}{n^{\frac{k+1}{2}-\delta}}\right)\right]\,,
\]
\noindent as $n\to \infty$. Define $\eta_k\coloneqq(2\pi)^{-\frac{k-1}{2}}\det\tilde{g}''(0)^{-\frac{1}{2}}$.
Similarly as in \cite[Proposition 2.7]{Pan}, one verifies that $\det\tilde{g}''(0) = \frac{(\omega k-2)^{k-1}}{(\omega k)^{k-2}}\neq 0$. 
Then
\begin{align*}
d(n,\delta,\omega) & = (\omega k-1)^{kn} \left[\frac{\eta_k\,L_0(\tilde{u}_0,\tilde{g})\,n^\delta}{\delta!\,n^{\frac{k-1}{2}}} + O\left(\frac{1}{n^{\frac{k+1}{2}-\delta}}\right)\right] \\
&= (\omega k-1)^{kn} \left[\frac{\eta_k\,L_0(\tilde{u}_0,\tilde{g})\,n^\delta}{\delta!\,n^{\frac{k-1}{2}}} + \frac{1}{n^{\frac{k-1}{2}-\delta}}\,O\left(\frac{\eta_k\,L_0(\tilde{u}_0,\tilde{g})}{n}\right)\right] \\
&= (\omega k-1)^{kn} \left[\frac{\eta_k\,L_0(\tilde{u}_0,\tilde{g})\,n^\delta}{n^{\frac{k-1}{2}}} + \frac{\eta_k\,L_0(\tilde{u}_0,\tilde{g})}{n^{\frac{k-1}{2}-\delta}}\,O\left(\frac{1}{n}\right)\right]\\
&= \eta_k\,L_0(\tilde{u}_0,\tilde{g})\frac{(\omega k-1)^{kn}}{n^{\frac{k-1}{2}-\delta}} \left[1+O\left(\frac{1}{n}\right)\right]\,.
\end{align*}
Conclusion follows by plugging in the identity for $L_0(\tilde{u}_0,\tilde{g})$ in Proposition \ref{prop: L_0}.
\end{proof}

\begin{Remark}
When $\omega=1$, we obtain the following $O(1/n)$-approximations for $d(n,\delta,1)$:
\[
d(n,\delta,1) \approx 
\begin{cases}
\frac{2}{\sqrt{3}\pi}\left(\frac{3}{2}\right)^\delta 8^n\,n^{\delta-1} & \mbox{if $k=3$}\\
\frac{27}{2^9\pi\sqrt{\pi}}\left(\frac{4}{3}\right)^\delta 81^{n}\,n^{\delta-\frac{3}{2}} & \mbox{if $k=4$}\,.
\end{cases}
\]
\end{Remark}

\section{Kalman varieties of partially symmetric singular vector \texorpdfstring{$k$}{k}-tuples}\label{sec: kalmanofsymmetricvectortuples}

In this final section, we restrict to the case $V_1=\cdots=V_k=V$ for some $n$-dimensional complex vector space $V$ and we assume that all isotropic quadrics $Q_i$ coincide with the quadric $Q\subset\PP(V)$. We also use the notations $\PP = \PP(V)^{\times k}$ and $Q^k = \prod_{i=1}^k Q\subset\PP$.
Our goal is to introduce another kind of Kalman variety. To motivate this construction, let us consider the matrix case. In this context, one may wonder what is the closure of the locus of matrices $A\in V\otimes V$ possessing a non-isotropic singular vector pair of the form $(\xx,\xx)$ for some $\xx\in V$. For $n=2$, this locus coincides with the subspace $S^2V\subset V\otimes V$ of $2\times 2$ symmetric matrices. As we shall see, for $n\ge 3$ this locus is not a linear subspace and contains the subspace $S^2V$ as a proper closed subset. 
\begin{Definition}
The {\it normalized symmetric Kalman variety} is the variety
\begin{equation}
\kappa^{nor}_{n,k} \coloneqq \overline{\{T\in\PP(V^{\otimes k})\mid \mbox{$T$ has a singular vector $k$-tuple $([\xx],\ldots,[\xx])$ for some $[\xx]\in \PP(V)\setminus Q$}\}}\,.
\end{equation}
In complete analogy with the generalized Kalman variety introduced in Definition \ref{def: genKalman}, one can define the {\em generalized symmetric Kalman variety}
\begin{equation}
\kappa_{n,k}\coloneqq \{T\in\PP(V^{\otimes k})\mid \mbox{$T$ has a singular vector $k$-tuple $([\xx],\ldots,[\xx])$ for some $[\xx]\in \PP(V)$}\}\,.
\end{equation}
As shown in Example \ref{ex: deg7example}, the generalized symmetric Kalman variety $\kappa_{n,k}$ may be reducible already for matrices.
\end{Definition}

\begin{Theorem}\label{thm: codim Kalman symmetric tuples}
The variety $\kappa^{nor}_{n,k}$ is irreducible of codimension $(k-1)(n-1)$ in $\PP(V^{\otimes k})$.
\end{Theorem}

\begin{proof}
Consider the spectral variety $\Sigma_{n{\bf 1},{\bf 1}}$ along with the two projections $\alpha$ and $\beta$ of $\Sigma_{n{\bf 1},{\bf 1}}$ onto $\PP(V^{\otimes k})$ and $\PP$, as in diagram \eqref{eq: diagram spectral variety}. 
By Theorem \ref{thm: fibers}, every fiber of $\beta$ at each point of $\PP\setminus Q^k$ is a linear subspace of codimension $k(n-1)$ in $\PP(V^{\otimes k})$. 

\smallskip
\noindent Consider the Segre embedding $\sigma\colon\PP\to\PP(V^{\otimes k})$, $\sigma([\xx_1],\ldots,[\xx_k])=[\xx_1\otimes\cdots\otimes\xx_k]$. 
Denote by $v_{n,k}$ the $k$-th Veronese embedding of $\PP(V)$, as a subvariety of $\PP(V^{\otimes k})$. The variety $\kappa^{nor}_{n,k}$ has the following description in terms of the diagram mentioned above: 
\[
\kappa^{nor}_{n,k}=\alpha(\overline{\beta^{-1}(\sigma^{-1}(v_{n,k})\setminus Q^k)}), 
\]
where $\sigma^{-1}(v_{n,k})=\PP\cap\Delta$ and $\Delta$ denotes the diagonal.
Let $Y = \overline{\beta^{-1}(\sigma^{-1}(v_{n,k})\setminus Q^k)}$. By Lemma \ref{lem: components}, $Y$ is irreducible and so is $\kappa^{nor}_{n,k}$. 
Moreover, we have $\dim(Y) = \dim(\PP\cap\Delta)+\dim(\PP(V^{\otimes k}))-(n-1)k = n^k-1-(k-1)(n-1)$. 

The general fiber of $\alpha$ restricted to $Y$ is finite.
Indeed, any point of $\sigma^{-1}(v_{n,k})\setminus Q^k$ is the equivalence class of a tuple of the same non-isotropic vector. The group $\mathrm{SO}(V)^{\times k}$ acts transitively on equivalence classes of non-isotropic vectors in $\PP$. Thus a general tensor in the fiber under $\beta$ of any point in the diagonal $\sigma^{-1}(v_{n,k})\setminus Q^k$ has a finite number of singular vector $k$-tuples. 
This implies the finiteness of $\alpha$ restricted to $Y$. The variety $\kappa^{nor}_{n,k}$ is irreducible and has dimension $\dim(Y)$. Hence $\kappa^{nor}_{n,k}$ has codimension $(k-1)(n-1)$ in $\PP(V^{\otimes k})$. 
\end{proof}

\begin{Example}\label{ex: deg7example}
Consider $k=2$, $V=\C^n$ and $Q=\mathcal{V}(x_1^2+\cdots+x_n^2)\subset\PP(V)$.
We have $\mathrm{codim}(\kappa^{nor}_{n,2})=n-1$. In particular, one finds $\kappa^{nor}_{2,2}=\PP(S^2\C^2)$. The first non-trivial case is the normalized symmetric Kalman variety $\kappa^{nor}_{3,2}$ which is a subvariety of $\PP(\C^3\otimes\C^3)$ of codimension $2$ and degree $7$. 
By Remark \ref{rmk: equal singular values}, $\kappa^{nor}_{3,2}$ coincides with the locus of matrices admitting a usual algebraic singular vector pair.
This ideal is generated by three cubics that may be recovered with the following construction, suggested to us by Jan Draisma.

Let $A=(a_{ij})\in\C^n\otimes\C^n$ and suppose $(\yy,\yy)$ is a (non-isotropic) singular vector pair of $A$.
Up to scaling $\yy$, we may write $A = \yy\yy^T + C$ and $A^T = \yy\yy^T+C^T$. The matrix $B=A-A^{T}$ has matrix rank at most $n-1$. Note that, since the latter matrix is skew-symmetric,  when $n$ is odd and for general $A$, one has $\mathrm{rk}(B) = n-1$. 
In such a case, the cofactor matrix $\mathrm{cof}(B)$ has rank one and satisfies $B\cdot \mathrm{cof}(B)^{T}=0$. So the image of $\mathrm{cof}(B)$ is spanned by the vector $\yy$ above. 

To derive the cubic equations for $n=3$, given $\xx = (x_1,x_2,x_3)^T\in \CC^3$, we have $\yy=\mathrm{cof}(B)\xx$. Since $\yy$ is then an eigenvector of $A$, let $\mathcal{K}$ be the ideal of $2\times 2$ minors of the matrix $[A\yy\mid\yy]$. We verified symbolically in \texttt{Macaulay2} that $\mathcal{K}=\mathcal{I}(\kappa^{nor}_{3,2})\cap\mathcal{J}^2$, where $\mathcal{I}(\kappa^{nor}_{3,2})$ is the ideal of $\kappa^{nor}_{3,2}$ generated by the desired three cubics, and where $\mathcal{J}=\langle x_3a_{12}-x_2a_{13}-x_3a_{21}+x_1a_{23}+x_2a_{31}-x_1a_{32}\rangle$ defines the hyperplane of matrices orthogonal (with respect to the Frobenius inner product) to the skew-symmetric matrix
\[
\begin{pmatrix}
0 & x_3 & -x_2 \\
-x_3 & 0 & x_1 \\
x_2 & -x_1 & 0
\end{pmatrix}\,.
\]
Moreover, $\kappa^{nor}_{3,2}$ is arithmetically Cohen-Macaulay and its reduced singular locus is $\PP(S^2\C^3)$.
The generalized symmetric Kalman variety $\kappa_{3,2}$ is reducible of degree $15$ and codimension $2$. One of its irreducible components is $\kappa^{nor}_{3,2}$.
Further numerical data for the varieties $\kappa^{nor}_{n,2}$ are summarized in Table \ref{table: degrees Kalman square matrices symmetric singular vector pair} for small values of $n$.

\begingroup
\setlength{\tabcolsep}{10pt}
\renewcommand{\arraystretch}{1.1}
\begin{table}[htbp]
\centering
\begin{tabular}{|c||c|c|c|}
\hline
$n$ & $\mathrm{codim}(\kappa^{nor}_{n,2})$ & $\deg(\kappa^{nor}_{n,2})$ & generators of the ideal $\mathcal{I}(\kappa^{nor}_{n,2})$ \\
\hhline{|=||===|}
2 & 1 & 1 & 1 linear\\
3 & 2 & 7 & 3 cubics\\
4 & 3 & 24 & 1 quadric, 1 quartic, 11 sextics\\
5 & 4 & 86 & 5 quartics, 5 quintics, 5 sextics, 31 septics\\
6 & 5 & 314 & 1 cubic, 1 quintic, 86 sextics\\
\hline
\end{tabular}
\vspace*{1mm}
\caption{Codimension and degree of $\kappa^{nor}_{n,2}$ for small $n$.}
\label{table: degrees Kalman square matrices symmetric singular vector pair}
\end{table}
\endgroup
\end{Example}

\begin{Example}
In the binary tensor case $(n=2)$, we obtain that $\mathrm{codim}(\kappa^{nor}_{2,k})=k-1$. The first non-trivial case is $\kappa^{nor}_{2,3}$ which is a subvariety of $\PP(\C^2\otimes\C^2\otimes\C^2)$ of codimension $2$ and degree $5$. Its ideal is generated by one quadric and four quartics.
\end{Example}

\begin{Question}
Table \ref{table: degrees Kalman square matrices symmetric singular vector pair} reports the degrees of the varieties $\kappa^{nor}_{n,2}$ for small values of $n$. What is the degree of $\kappa^{nor}_{n,k}$ in general?
\end{Question}

It is clear that the construction above can be carried out taking into account partial symmetries of singular vector $k$-tuples. For simplicity, we remain in the hypercubical format $n^{\times k}$. To this aim, let $\bomega=(\omega_1,\ldots,\omega_t)$ be a partition of $k$, namely $\omega_i\ge 1$ for all $i$ and $\omega\coloneqq \omega_1+\cdots+\omega_t=k$. Without loss of generality, we assume that $\omega_1\ge\cdots\ge\omega_t$. We denote by $s_{n,k}^\bomega$ the Segre-Veronese embedding of $\PP(V)^{\times t}$ in $\PP(S^{\bomega}V)$, and we consider it as a subvariety of $\PP(V^{\otimes k})$. Define
\begin{equation}
\kappa^{nor}_{n,k,\bomega} \coloneqq \overline{  \{T\in\PP(V^{\otimes k})\mid \mbox{$T$ has a singular vector $k$-tuple $(\underbrace{[\xx_1],\ldots,[\xx_1]}_{\omega_1},\ldots,\underbrace{[\xx_t],\ldots,[\xx_t]}_{\omega_t})\in \PP\setminus Q^k$}\}}.
\end{equation}

\noindent With similar arguments as in the proof of Theorem \ref{thm: codim Kalman symmetric tuples}, one proves the following result.

\begin{Theorem}\label{thm: codim Kalman partially symmetric tuples}
Let $\bomega=(\omega_1,\ldots,\omega_t)$ be a partition of $k$.
The variety $\kappa^{nor}_{n,k,\bomega}$ is irreducible of codimension $(k-t)(n-1)$ in $\PP(V^{\otimes k})$. So the codimension depends only on the number of parts $t$ of $\bomega$.
\end{Theorem}

\begin{Example}
In the binary tensor case $(n=2)$, we see that $\mathrm{codim}(\kappa^{nor}_{2,k,\bomega}) =k-t$. The first non-trivial case not considered before is $\kappa^{nor}_{2,3,(2,1)}$ which is a hypersurface of $\PP(\C^2\otimes\C^2\otimes\C^2)$ of degree $8$. If $\bomega=(2,1^{k-2})$, we have $t=k-1$ and therefore $\kappa^{nor}_{2,k,\bomega}$ is always a hypersurface. Moreover, note that, for instance, $\mathrm{codim}(\kappa^{nor}_{2,4,(2,2)})=\mathrm{codim}(\kappa^{nor}_{2,4,(3,1)})=2$, even though $\kappa^{nor}_{2,4,(2,2)}$ and $\kappa^{nor}_{2,4,(3,1)}$ have different degrees.
\end{Example}
\bibliographystyle{alpha}
\bibliography{biblio}
\end{document}